\documentclass{amsart}

\usepackage{url}

\usepackage{amsmath}
\usepackage{amsfonts}
\usepackage{amssymb}
\usepackage{xcolor}
\usepackage{graphicx}
\newcommand{\Second}{\textup{I}\!\textup{I}}

\newtheorem{theorem}{Theorem}

\newtheorem{definition}{Definition}
\newtheorem{lemma}{Lemma}
\newtheorem{proposition}{Proposition}
\newtheorem{remark}{Remark}

\begin{document}

\title[Covariance matrix, EIG and LLE]{Connecting dots -- from local covariance to empirical intrinsic geometry and locally linear embedding}

\author{John Malik}
\address{Department of Mathematics, Duke University, Durham, NC, USA}
\email{john.malik@duke.edu}

\author{Chao Shen}
\address{Department of Mathematics, Duke University, Durham, NC, USA}
\email{chao.shen@duke.edu}

\author{Hau-Tieng Wu}
\address{Department of Mathematics and Department of Statistical Science, Duke University, Durham, NC, USA; Mathematics Division, National Center for Theoretical Sciences, Taipei, Taiwan}
\email{hauwu@math.duke.edu}

\author{Nan Wu}
\address{Department of Mathematics, Duke University, Durham, NC, USA}
\email{nan.wu@duke.edu}

\begin{abstract}
Local covariance structure under the manifold setup has been widely applied in the machine learning society. Based on the established theoretical results, we provide an extensive study of two relevant manifold learning algorithms, empirical intrinsic geometry (EIG) and the locally linear embedding (LLE) under the manifold setup. Particularly, we show that without an accurate dimension estimation, the geodesic distance estimation by EIG might be corrupted. Furthermore, we show that by taking the local covariance matrix into account, we can more accurately estimate the local geodesic distance. When understanding LLE based on the local covariance structure, its intimate relationship with the curvature suggests a variation of LLE depending on the ``truncation scheme''. We provide a theoretical analysis of the variation.
\end{abstract}

\maketitle

\section{Introduction}\label{Sect:Introduction}

Covariance is arguably one of the most important quantities in data analysis. It has been widely studied in the past century and is still an active research topic nowadays. In this paper, we focus on the {\em local covariance structure} under the manifold setup, which has been widely applied, explicitly or implicitly, to various applications in different fields; see, for example, a far-from-complete list \cite{KambhatlaLeen1997,Roweis_Saul:2000,Donoho_Grimes:2003,Brand2003,Zhang_Zha:2004,Kushnir2006,Goldberg2009,Salhov2012,Gong2012,Singer_Wu:2012,Pedagadi2013,Little2017,Arias-Castro2017}. In the past few years, its mathematical and statistical properties have been well established \cite{Singer_Wu:2012,Cheng_Wu:2013,Bernstein2014, Kaslovsky2014, Tyagi2013,wu2017think} for different purposes. 
In this paper, based on the established theoretical foundation, we extensively discuss two topics in the manifold learning society {that are related to the local covariance structure} -- empirical intrinsic geometry (EIG) and locally linear embedding (LLE).

{EIG \cite{Talmon2012,Talmon2013}, or originally called non-linear independent component analysis \cite{Singer2008}, is a technique aiming to deal with the distortion underlying the collected dataset that is caused by the observation process. In many applications, the manifold structure we have interest cannot be directly accessed, but only via an observation. However, the observation process might nonlinearly deform the manifold we have interest. As a result, the information inferred from the observed data point cloud might not faithfully reflect the intrinsic properties.
The goal of EIG is correcting this deformation by taking the local covariance matrix into account. From the statistical viewpoint, it is a {nonlinear {\em latent space model}, and the local covariance structure leads to a} generalization of the Mahalanobis distance. While it has been successfully applied to different problems \cite{Wu_Talmon_Lo:2015,Mishne2015,YairTalmon2017,Shemesh2017,2018arXiv180301710L}, to the best of our knowledge, {except an argument on the Euclidean space setup \cite{Singer2008},} a systematic evaluation of how the algorithm works under the manifold setup, and its sensitivity to the parameter choice, is missing. Due to its importance, the first contribution of this paper is providing a quantification of EIG under the manifold setup, and discuss how the chosen parameter influences the final result. In the special case that there is no deformation (that is, we can access the manifold directly), we show a more accurate geodesic distance estimator, {called the {\em covariance-corrected geodesic distance estimator},} by correcting the Euclidean distance when the manifold is embedded in the Euclidean space.

LLE \cite{Roweis_Saul:2000} is a widely applied nonlinear dimension reduction technique in the manifold learning society. Despite its wide application, its theoretical properties were studied only until recently. See \cite{wu2017think} as an example. Based on the analysis, several peculiar behaviors of LLE have been better understood. While LLE depends on the barycentric coordinate to determine the affinity between pairs of points, it has a natural relationship with the local covariance matrix; the kernel associated with LLE is not symmetric, which is different from the kernel commonly used in graph Laplacian-based algorithms like Laplacian Eigenmaps \cite{belkin2003} or Diffusion Maps \cite{Coifman2006}. 
The regularization plays an essential role in the algorithm. Different regularizations lead to different embedding results. Based on the intimate relationship between the curvature and regularization, the second contribution of this paper is studying a variation of LLE by directly truncating the local covariance matrix.

{The paper is organized in the following way. In Section \ref{Section LCM}, we introduce the notation for the local covariance structure analysis and some relevant known results.}
In Section \ref{Sect:Other}, we provide a theoretical argument of EIG under the manifold setup; when the observation process is trivial, we {analyze the} covariance-corrected geodesic distance estimator. In Section \ref{Sect:Other2}, we discuss the relationship of EIG and LLE, and provide a variation of LLE. Numerical results are shown in each section to support the theoretical findings. 
{In Section \ref{Section Numerics}, we provide some numerical results.}
In Section \ref{Sect:Conclusion}, discussion and conclusion are provided.

\subsection{Notations and Mathematical Setup}\label{Sect:Prelims}

Let $X$ be a $p$-dimensional random vector with the range supported on a $d$-dimensional, compact, smooth Riemannian manifold $(M, g)$ isometrically embedded in $\mathbb{R}^p$ via $\iota:M\hookrightarrow \mathbb{R}^p$.  We assume that $M$ is boundary-free in this work. Let $\exp_{x}:T_xM\to M$ be the exponential map at $x$. Unless otherwise specified, we will carry out calculations using normal coordinates. 
Let $T_xM$ denote the tangent space at $x\in M$, and let $\iota_*T_{x}M$ denote the embedded tangent space in $\mathbb{R}^p$.  Write the normal space at $y=\iota(x)$ as $(\iota_*T_{x}M)^\bot$. 
Let $\Second_{x}$ be the second fundamental form of $\iota$ at $x$.  Let $P$ be the probability density function (p.d.f.) associated with the random vector $X$ \cite[Section 4]{Cheng_Wu:2013}. We assume that $P \in \mathcal{C}^5(\iota(M))$ and that there exists $0 < P_m \leq P_M$ such that $P_m \leq P(y) \leq P_M < \infty$ for all $y\in \iota(M)$. Let $\mathbf{e}_i\in\mathbb{R}^p$ be the unit $p$-dimensional unitary vector with an $1$ in the $i$th entry. Let $B_h^{\Bbb R^p}(z)$ denote the $p$-dimensional Euclidean ball of radius $h > 0$ with centre $z \in \Bbb R^p$, and let $\chi_{A} \colon \Bbb R^p \rightarrow \{0, 1\}$ denote the indicator function of the set $A \subset \Bbb R^p$.
{Denote $\mathbb{O}(p)$ to be the orthogonal group in dimension $p\in \mathbb{N}$.}
We define two main quantities, the {\em truncated inverse}  and {\em regularized inverse}, of a symmetric matrix that are related to EIG and LLE respectively.

\begin{definition} \label{trunC}
Let $A \in \mathbb{R}^{p \times p}$ be a real symmetric matrix. $r=\textup{rank}(A)$. 
Let $\lambda_{1} \geq \ldots  \geq \lambda_{p}$ be the eigenvalues of $A$, and let  $u_{1}, \ldots,  u_{p}$ be the corresponding normalized eigenvectors. For $0<\alpha \leq r$, the $\alpha$-truncated inverse of $A$ is defined as 
\begin{align}
\mathcal{T}_\alpha[A] =&\, \begin{bmatrix}
u_{1} & \cdots & u_{\alpha}
\end{bmatrix} \begin{bmatrix} \lambda_{1}^{-1} & & \\
 & \ddots & \\
 & & \lambda_{\alpha}^{-1} 
 \end{bmatrix} 
 \begin{bmatrix}
u_{1}^\top\\
\vdots\\
u_{\alpha}^\top
\end{bmatrix}.
\end{align}
Choose a regularization constant $c> 0$. The $c$-regularized inverse of $A$ is defined as
\begin{equation}
\mathcal{I}_c[A] = \begin{bmatrix}
u_{1} & \cdots & u_{r}
\end{bmatrix} \begin{bmatrix} (\lambda_{1} + c )^{-1} & & \\
 & \ddots & \\
 & & (\lambda_{r} + c )^{-1} 
 \end{bmatrix} 
 \begin{bmatrix}
u_{1}^\top\\
\vdots\\
u_{r}^\top
\end{bmatrix},
\end{equation}
\end{definition}
Note that if {$\alpha=r$}, then $\mathcal{T}_\alpha[A]$ is the Penrose-Moore pseudo-inverse. When $c\to 0$, $\mathcal{I}_c[A]$ becomes $\mathcal{T}_r[A]$.

\section{Local Covariance Matrix and Some Facts}\label{Section LCM}
{We start with the definition of the local covariance matrix.}
\begin{definition}
For $x \in M$ and a measurable set $\mathcal{O}\subset \iota(M)$, the {\em local covariance matrix at $\iota(x)\in \iota(M)$ associated with $\mathcal{O}$} is defined as 
\begin{equation} \label{general def of covariancematrix}
 C_{x,\mathcal{O}}:=\mathbb{E}[(X-\iota(x))(X-\iota(x))^{\top}\chi_{\mathcal{O}}(X)]\in\mathbb{R}^{p\times p}.
\end{equation}
When $\mathcal{O}$ is $B_{h}^{\mathbb{R}^p}(\iota(x))\cap \iota(M)$, where $h>0$, we denote 
\[
C_h(x):= C_{x,B_{h}^{\mathbb{R}^p}(\iota(x))\cap  \iota(M)}
\]
and simply call $ C_h(x)$ the {\em local covariance matrix at $\iota(x)$}. 
\end{definition}
The local covariance matrix and its relationship with the embedded tangent space of the manifold have been widely studied recently, including (but not exclusively) \cite{Singer_Wu:2012,Cheng_Wu:2013,Tyagi2013,Bernstein2014, Kaslovsky2014,Little2017}. Recently, in order to systematically study the LLE algorithm, the higher order structure of the local covariance matrix was explored in \cite{wu2017think}. We now summarize the result for our purpose. 
Since the local covariance matrix is invariant up to translation and rotation and the analysis is local at one point, to simplify the discussion, from now on, when we analyze the local covariance matrix at $x$, we always assume that the manifold is translated and rotated in $\mathbb R^p$ so that $\iota_*T_xM$ is spanned by $\mathbf{e}_1,\ldots,\mathbf{e}_d \in \Bbb R^p$.

\begin{lemma}(\cite[Proposition~3.2]{wu2017think}) \label{prop1}
Fix $x\in M$ and take $h>0$. Write the eigen-decomposition of $ C_{h}(x)=U_{h}(x)\Lambda_{h}(x) U_{h}(x)^{\top}$.
Then, when $h$ is sufficiently small, the diagonal matrix $\Lambda_{h}(x)\in\mathbb{R}^{p\times p}$ satisfies
\begin{align}\label{Equation:Cov_EVal}
\Lambda_{h}(x) &=\frac{|S^{d-1}|{P}(x)h^{d+2}} {d(d+2)}  \begin{bmatrix}
I_{d\times d} +O(h^2)  & 0 \\
0 & O(h^2) \\
\end{bmatrix},
\end{align}
and the orthogonal matrix $U_{h}(x) \in {\mathbb{O}(p)}$ satisfies
\begin{align}
U_{h}(x)&=\begin{bmatrix} U_1 & 0 \\ 0 & U_2\end{bmatrix}(I_{p\times p}+h^2 \mathsf{S})+O(h^4),\label{Expansion:U_x}
\end{align}
where $U_1\in {\mathbb{O}(d)}$, $U_2\in  {\mathbb{O}(p-d)}$, and $\mathsf{S}$ is an anti-symmetric matrix. 
\end{lemma}

First, note that the local covariance matrix depends on the p.d.f.. Particularly, when the sampling is non-uniform, the eigenvalues are deviated by the p.d.f..
Equation (\ref{Equation:Cov_EVal}) in this lemma says that the first $d$ eigenvalues of $ C_{h}(x)$ are of order $h^{d+2}$, while the rest of the eigenvalues are two orders higher; that is, of order $O(h^{d+4})$. Moreover, note that $I_{p\times p}+h^2 \mathsf{S}$ approximates a rotation, so equation (\ref{Expansion:U_x}) says that the first $d$ normalized eigenvectors of $ C_{h}(x)$ deviate from an orthonormal basis of $\iota_*T_xM$ by a rotational error of order $O(h^2)$, and the rest of the orthonormal eigenvectors of $ C_{h}(x)$ deviate from an orthonormal basis of $(\iota_*T_xM)^\bot$ within a rotational error of order $O(h^2)$. 

In practice, we are given a finite sampling of points from the embedded manifold and need to approximate the local covariance matrix. 
Since the finite convergence argument of the sample local covariance matrix to the local covariance matrix is standard (see, for example, \cite[Propositions 3.1 and 3.2 and Lemma~E.4]{wu2017think}), to focus on the main idea and simplify the discussion, below we work directly on the continuous setup; that is, we consider only the asymptotical case when $n\to \infty$.

It has been well known that when a manifold $M$ is isometrically embedded into another manifold $M'$, then for two close points $x,y\in M$, the geodesic distance between $x,y$ in $M$ could be well approximate by the geodesic distance between $x,y$ in $M'$, with the error depending on the second fundamental form of the embedding (see, for example \cite[Proposition 6]{Smolyanov:2007}). 
We have the following lemma when $M'$ is Euclidean space.
\begin{lemma}(\cite[Lemma B.2]{wu2017think}) \label{Lemma:1}
Suppose that $M$ is isometrically embedded in $ \mathbb{R}^p$ through $\iota$. Fix $x\in M$ and use polar coordinates $(t,\theta)\in [0,\infty)\times S^{d-1}$ to parametrize $T_{x}M$. For $y=\exp_x(\theta t)$ for sufficiently small $t$, we have:
\begin{align}\label{oldLemma1}
\iota(y) -\iota(x)
=&\,(\iota_{*}\theta) t+\frac{\Second_x(\theta,\theta) }{2}t^2+\frac{\nabla_{\theta}\Second_x(\theta,\theta)}{6}t^3+O(t^4).
\end{align}
Moreover, when $h:=\|\iota(y)-\iota(x)\|_{\mathbb{R}^p}$ is sufficiently small, we have
\begin{align}
h &=  t-\frac{\|\Second_x(\theta,\theta)\|^2}{24} t^3 - \frac{\nabla_{\theta}\Second_x(\theta,\theta)  \cdot  \Second_x(\theta,\theta)}{24}t^4 +O(t^5)\,.\label{oldLemma2a} 
\end{align}
and hence
\begin{align}
t = h + \frac{\|\Second_x(\theta,\theta)\|^2}{24}h^{3} +\frac{\nabla_{\theta}\Second_x(\theta,\theta)  \cdot  \Second_x(\theta,\theta)}{24}h^4 +O(h^5). \label{oldLemma2b}
\end{align}
\end{lemma}
The proof of the Lemma could be found in, for example, in \cite[Proposition 6]{Smolyanov:2007} when $M'$ is a generic manifold or \cite[Lemma B.2]{wu2017think} when $M'$ is Euclidean space.

\section{Empirical Intrinsic Geometry and Local Covariance Matrix}\label{Sect:Other}

EIG \cite{Talmon2012,Singer2008} is a technique aiming to deal with the underlying distortion caused by the observation process. The basic idea of the technique is that the local covariance matrix captures the distortion, under suitable assumptions, and we can correct the distortion by manipulating the local covariance matrix and recover the local geodesic distance. From the statistical viewpoint, it is a generalization of the Mahalanobis distance. We now examine the {intimate} relationship between EIG and the covariance-corrected geodesic estimator.

Suppose that $M$ is a $d$ dimensional closed Riemannian manifold that hosts the information we have interest but we cannot directly access. {We assume that there is a method to {\em indirectly} access $M$ via an {\em observation}, and hence collect a dataset that contains {\em indirect} information of $M$ that we have interest.}
We model the observation as a {nonlinear function} $\mathsf{\Phi}: M \to N$, where $\mathsf{\Phi}$ is a diffeomorphism {of $M$} and $N$ is isometrically embedded in $\mathbb{R}^q$ via $\iota$. Under this setup, for the point $x\in M$ that we cannot access, there is a corresponding point $\iota(y)\in\iota(N)\subset \mathbb{R}^q$ that we can access and collect as the dataset, where $y=\mathsf{\Phi}(x)$. The mission is estimating the geodesic distance between two close points $x\in M$ and $w\in M$ through accessible data points $\iota(y)=\iota(\mathsf{\Phi}(x))$ and $\iota(z)=\iota(\mathsf{\Phi}(w))$. This situation is commonly encountered in data analysis; for example, the brain activity information contained in electroencephalogram {recorded from the scalp might be deformed due to the process of recording the electroencephalogram, the anatomical structure and physiological properties.
Clearly, due to the diffeomorphism associated with the observation, the pairwise distance estimated from the collected database no longer faithfully reflects the pairwise distance of the inaccessible space. Hence, analysis tools depending on the pairwise distance are biased}. The interest of EIG is correcting this bias. 

The main idea of EIG can be summarized in the following way. For each hidden point $x\in M$, take the geodesic ball $B_{\epsilon}(x)\subset M$ with the radius $\epsilon$ and centered at $x$.
As is discussed in \cite{Singer2008}, 
{\em if we can determine} $E(y):=\mathsf{\Phi}(B_{\epsilon}(x))\subset N$ around $y$ that is associated with $B_{\epsilon}(x)$, then up to a constant, the geodesic distance between two close points $w$ and $x$ can be well approximated by $\iota(y)$, $\iota(z)$, and the local covariance matrix $C_{y,\iota(E(y))}$ associated with $\iota(E(y))$ as defined in equation (\ref{general def of covariancematrix}). To simplify the notation, denote $\mathcal{C}_\epsilon(y):= C_{y,\iota(E(y))}$.  See Figure \ref{Fig:illustration} for an illustration of the setup. 
Note that when $M$ is Euclidean and $\mathsf{\Phi}$ is linear, $E(y)$ is an ellipsoid. In the following, although we consider the manifold model that is in general not Euclidean, we abuse the terminology and call $E(y)$ an ellipsoid. 
Numerically, $\mathcal C_{\epsilon}(y)$ is estimated by {taking points in ${E}(y)$ into account. }Again, since the convergence proof is standard, to simplify the discussion, we skip the finite convergence step and focus on the continuous setup.

\begin{figure} 
\centering
\includegraphics[width=0.95\textwidth]{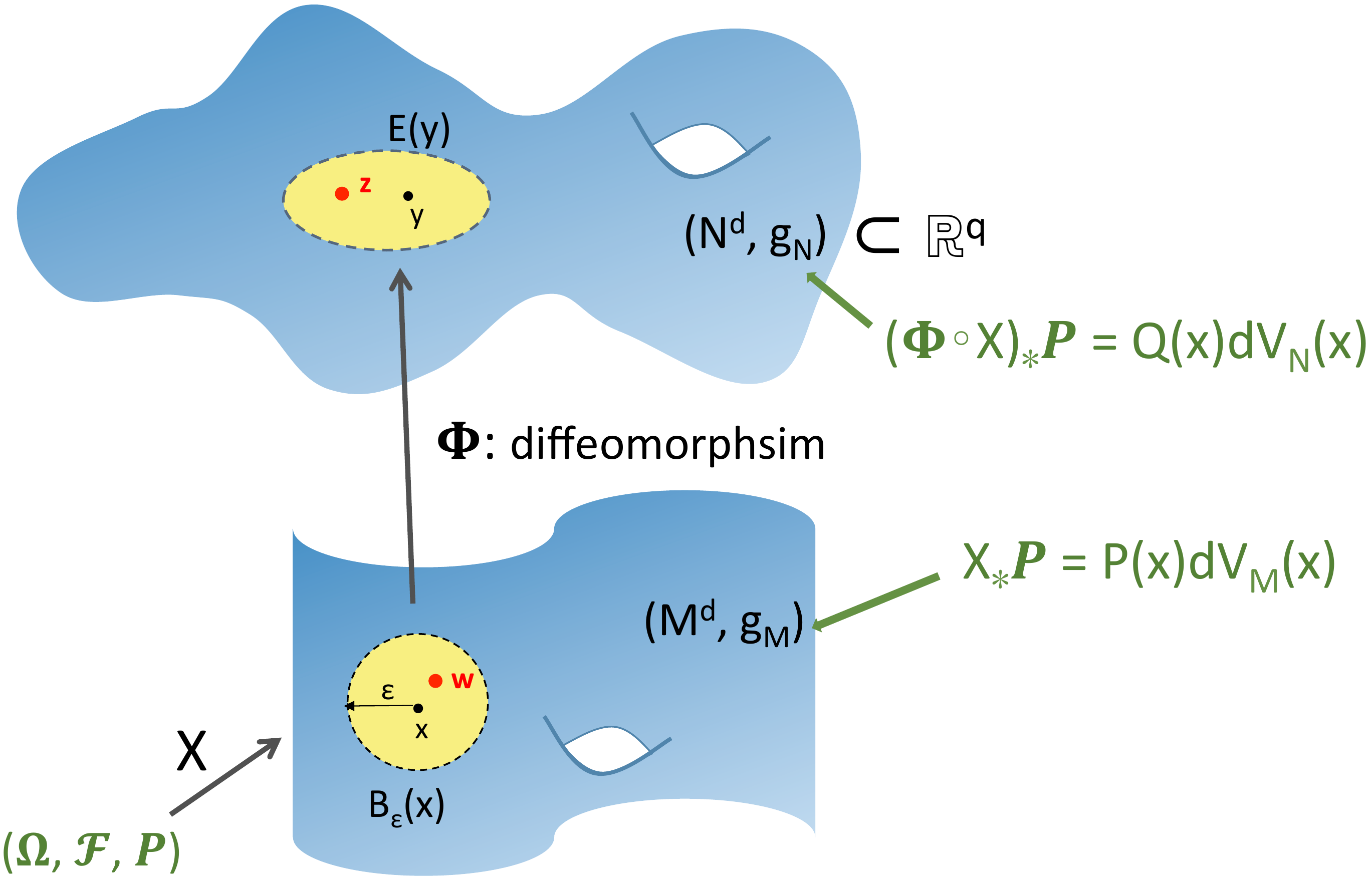}
\caption{The illustration of the theoretical framework of EIG.}\label{Fig:illustration}
\end{figure}

This main idea is carried out by noting that the local covariance matrix at $y$ associated with $\mathcal{C}_{\epsilon}(y)$ captures the Jacobian of the diffeomorphism associated with the observation, that is, $\mathcal C_{\epsilon}(y)\approx \nabla \mathsf{\Phi}|_{x}\mathsf{\Phi}|_{x}^\top$ \cite{Singer2008}.
Under the assumption that the ellipsoid $E(y)$ is known, $\mathcal C_{\epsilon}(y)\approx \nabla \mathsf{\Phi}|_{x}\mathsf{\Phi}|_{x}^\top$ is true and the dimension of the manifold $d$ is known, authors in \cite{Singer2008,Talmon2012} consider the following quantity, called the {\em EIG distance}, to estimate the geodesic distance between $x$ and $w$:
\begin{gather}\label{EIG}
 (\iota(z) - \iota(y))^\top  \left(\frac{\mathcal{T}_\alpha[\mathcal C_{\epsilon}(y)]+\mathcal{T}_\alpha[\mathcal C_{\epsilon}(z)]}{2}\right) (\iota(z) - \iota(y))\,,
\end{gather}
where $\alpha$ is chosen to be the dimension of $d$. 

\subsection{Challenges of EIG and analysis}

However, there are two challenges of applying the EIG idea: how to determine the ellipsoid ${E}(y)$ associated with $B_\epsilon(x)$, and how to estimate the dimension $d$ of the intrinsic manifold. In \cite{Singer2008} and most of its citations, the state space model combined with the stochastic differential equation (SDE) is considered, and the ellipsoid can be determined simply by taking the temporal relationship into account under this model. 
We refer readers with interest to \cite{Singer2008,Talmon2012} for more details about this state-space model.
If the state-space model and SDE cannot be directly applied, this task is most of time a big challenge and to the best of our knowledge not too much is known. This challenge is however out of the scope of this paper. 
On the other hand, most of time we do not have an access to the dimension of $d$, and a dimension estimation is needed. Although theoretically we can count on the spectral gap to estimate $d$, in practice it depends on the try-and-error process and little can be guaranteed. Particularly, when the manifold is highly distorted by the observation, faithfully estimating the dimension is another challenging task. 
To the best of our knowledge, although EIG has obtained several successes in different applications \cite{Wu_Talmon_Lo:2015,Mishne2015,YairTalmon2017,Shemesh2017}, a systematic exploration of the approximation under the general manifold setup is lacking. It is also not clear what may happen when the dimension is not estimated correctly.

Below, we assume that we have had the knowledge of the ellipsoid associated with $B_\epsilon(x)$ for all $x$, and hence we can evaluate the local covariance matrix at $y=\mathsf{\Phi}(x)$ associated with ${E}(y)$. We show how $\mathcal C_{\epsilon}(y)\approx \nabla \mathsf{\Phi}|_{x}\mathsf{\Phi}|_{x}^\top$ holds under the manifold setup, and the influenced of an erroneously estimated dimension. We start with the following definitions.
\begin{definition}
For $y\in N$, define the {\em normalized local covariance matrix} at $y$ associated with $E(y)\subset N$ as
\begin{align}
\bar{\mathcal C}_{\epsilon}(y):=\frac{\mathcal C_\epsilon(y)}{\epsilon^2\mathbb{E}[\chi_{E(y)}(Y)]},\label{Definition:NormalizedCovariance}
\end{align}
{where $Y:= \iota \circ \mathsf{\Phi} \circ X$ is the induced random variable.}
\end{definition}

The normalization step in (\ref{Definition:NormalizedCovariance}) is introduced to remove the impact of the nonuniform sampling. As we will see, without this normalization, when the sampling is non-uniform, the p.d.f. will play a role in the final analysis.

\begin{lemma}\label{Lemma2ofTheorem2}
Let $x \in M$ and $B_{\epsilon}(x) \subset M$ be the geodesic ball around $x$. The local covariance matrix at $y=\mathsf{\Phi}(x)$ associated with the ellipsoid $E(y)=\mathsf{\Phi}(B_{\epsilon}(x))\subset N$ satisfies
\begin{align}
\mathcal C_\epsilon(y) = \frac{|S^{d-1}|  P(x) \epsilon^{d+2}}{d(d+2)}[\iota_*|_y\nabla \mathsf{\Phi}(x)][ \nabla\mathsf{\Phi}(x)^\top \iota_*|_y^\top] +O(\epsilon^{d+4})\label{Expansion:Ceps:Lemma2}
\end{align}
and the normalized local covariance matrix at $y$ associated with $E(y)$ satisfies
\begin{align}
\bar{\mathcal  C}_{\epsilon}(y)=\frac{1}{d+2}[\iota_*|_y\nabla \mathsf{\Phi}(x)] [\nabla\mathsf{\Phi}(x)^\top \iota_*|_y^\top] +O(\epsilon^2)\label{Expansion:barCeps:Lemma2}\,.
\end{align}
If $v_1, v_2 \in (\iota_*T_{y}N)^\bot$, then 
\begin{align}
v_1^\top \bar{\mathcal  C}_{\epsilon}(y) v_2=\,& \frac{d\epsilon^2}{4|S^{d-1}|(d+4)}\int_{S^{d-1}}v_1^\top \Second_y(\nabla_{\theta} \mathsf{\Phi}(x), \nabla_{\theta} \mathsf{\Phi}(x)) (\Second_y(\nabla_{\theta} \mathsf{\Phi}(x), \nabla_{\theta} \mathsf{\Phi}(x)))^\top v_2 d\theta+O(\epsilon^{4})\,.\nonumber
\end{align}
\end{lemma}

By this lemma, we see that the expansion of $\mathcal C_\epsilon(y)$ depends not only on the p.d.f. but also the Jacobian of the deformation, and the normalization step cancels this dependence. The proof is postponed to Appendix, page \pageref{Proof:Lemma2ofTheorem2}. Based on this lemma and the considered setup in \cite{Singer2008,Talmon2012}, we consider the following quantity.

\begin{definition}
For $y,z\in N$, define the {\em EIG distance of order $\alpha$} between $y,z$, where $1\leq \alpha\leq q$, as
\begin{gather}\label{EIG cont}
\textsf{EIG}^2_{\alpha}(y,z) = (\iota(y) - \iota(z))^\top  \left(\frac{\mathcal{T}_\alpha[\bar{\mathcal C}_{\epsilon}(y)]+\mathcal{T}_\alpha[\bar{\mathcal C}_{\epsilon}(z)]}{2}\right) (\iota(y) - \iota(z))\,.
\end{gather}
\end{definition}
Note that this definition is slightly different from that considered in (\ref{EIG}). {Since the only difference is the normalization step, the result below can be directly translated for (\ref{EIG}).} Below we show that the EIG distance between $y=\mathsf{\Phi}(x)\in N$ and $z=\mathsf{\Phi}(w)\in N$ is a good estimator of the geodesic distance between $x\in M$ and $w\in M$ only when some conditions are satisfied.

\begin{theorem} \label{Theorem2}
Suppose two $d$-dimensional smooth closed Riemannian manifolds $M$ and $N$ are diffeomorphic via $\mathsf{\Phi}:M\to N$, and suppose $N$ is isometrically embedded in $\mathbb{R}^q$ via $\iota$. 
Take $x\in M$ and $w \in B_{\epsilon}(x) \subset M$, where $\epsilon>0$, and denote $t$ to be the geodesic distance between $x$ and $w$. Let $y=\mathsf{\Phi}(x)$, $z=\mathsf{\Phi}(w)$. Take $\alpha \leq \min(\texttt{rank} [\bar{\mathcal C}_{\epsilon}(y)], \texttt{rank} [ \bar{\mathcal C}_{\epsilon}(z)] )$. When $\epsilon$ is sufficiently small, we have the following: 
\begin{enumerate}
\item When $\alpha=d$, we have 
\begin{align}
\mathsf{EIG}_{d}(y,z)= \sqrt{d+2}t+ O(t\epsilon^2)+O(t^3) \,; \nonumber
\end{align}

\item When $1\leq \alpha<d$, we have
\begin{align}
\mathsf{EIG}_{\alpha}(y,z)=\sqrt{d+2}\,t +O(t)\,.
\end{align}
Let $V_\alpha(y)$ and $V_\alpha(z)$ be the subspaces of $\iota_*T_yN$ and $\iota_*T_zN$ generated by the first $\alpha$ eigenvectors of $\bar{\mathcal C}_{\epsilon}(y)$ and $\bar{\mathcal C}_{\epsilon}(z)$ respectively. Suppose $z=\exp_y\vartheta(y)$ for $\vartheta(y) \in T_yN$ and $y=\exp_z \vartheta(z)$ for $\vartheta(z) \in T_zN$. 
If $\iota_{*}|_y \vartheta(y) \in V_\alpha(y)$ and $\iota_{*}|_z \vartheta(z) \in V_\alpha(z)$, we have
\begin{align}
\mathsf{EIG}_{\alpha}(y,z)=\sqrt{d+2}\,t +O(t\epsilon^2)+O(t^3)\,.
\end{align}

\item When $\alpha>d$,  assume that the smallest nonzero eigenvalues of $\bar{\mathcal C}_{\epsilon}(y)$ and $\bar{\mathcal C}_{\epsilon}(z)$ are of order $\epsilon^4$ {and there are $l_y\geq0$ and $l_z\geq0$ such eigenvalues respectively. 
In general, we have
\begin{align}
\mathsf{EIG}_{\alpha}(y,z)=\sqrt{d+2}\,t+O(t)\,.
\end{align}
When $l_y=0$ or $\alpha\leq q-l_y$ hold, and $l_z=0$ or $\alpha\leq q-l_z$ hold, for $t=\epsilon^\beta$, where $\beta>1$, we have 
\begin{align}\label{ResultAlphagreatthanD}
\mathsf{EIG}_{\alpha}(y,z)&=\sqrt{d+2}\,t+O(t(t/\epsilon+\epsilon)^2)\\
&=\sqrt{d+2}\,t+O(t^{1+\min\{1-1/\beta,1/\beta\}})\,.\nonumber
\end{align}}
\end{enumerate}
\end{theorem}
The proof is postponed to Section \ref{Section:Appendix:Theorem2:proof}. 
We now have discussion of the theorem. 
In general, $\texttt{rank} [\bar{\mathcal C}_{\epsilon}(y)]$ and $\texttt{rank} [ \bar{\mathcal C}_{\epsilon}(z)]$ are different, thus we need the condition $\alpha \leq \min(\texttt{rank} [\bar{\mathcal C}_{\epsilon}(y)], \texttt{rank} [ \bar{\mathcal C}_{\epsilon}(z)] )$. 
First of all, if we know the dimension of the manifold, then the EIG distance of order $d$ between $y=\mathsf{\Phi}(x)\in N$ and $z=\mathsf{\Phi}(w)\in N$ is an accurate estimator of the geodesic distance between $x\in M$ and $w\in M$, up to a global constant $\sqrt{d+2}$. This result coincides with the claim in \cite{Singer2008,Talmon2012} when the sampling on $N$ is uniform. However, when the sampling on $N$ is non-uniform, the result is deviated by the p.d.f. if we replace the normalized local covariance matrix $\bar{\mathcal C}_{\epsilon}(y)$ in (\ref{EIG cont}) by the local covariance matrix $\mathcal C_{\epsilon}(y)$. This deviation can be seen from comparing (\ref{Expansion:Ceps:Lemma2}) and (\ref{Expansion:barCeps:Lemma2}) in Lemma \ref{Lemma2ofTheorem2}. 

Second, if the dimension is unknown and wrongly estimated, the result depends on the situation.
When the dimension is under-estimated, in general the estimator is wrong. In a non-generic situation that the geodesic direction from $y$ to $z$ and that from $z$ to $y$ are both located on the first $\alpha$ eigenvectors of the associated local covariant matrices, we may still obtain an accurate estimator. Note that due to the curvature, there is no guarantee that the first $\alpha$ eigenvectors of $\bar{\mathcal C}_{\epsilon}(y)$ will contain $\vartheta(y)$; even if it does, there is no guarantee that the first $\alpha$ eigenvectors of $\bar{\mathcal C}_{\epsilon}(z)$ will contain $\vartheta(z)$.
{When the dimension is over-estimated and some assumptions are satisfied, we still can obtain a reasonably well estimate of the intrinsic geodesic distance when $t/\epsilon$ is sufficiently small, but with a slow convergence rate, as is shown in \eqref{ResultAlphagreatthanD}. }

{Third, it is important to note that this estimate for the geodesic distance is valid only for points that are ``not too far away''. The estimate might be degenerate if two points are far away. For example, if $y\neq z\in N$ are two points so that the vector $\iota(z) - \iota(y)$ is orthogonal to the column space of $\mathcal{T}_\alpha[\mathcal{C}_\epsilon(y)] + \mathcal{T}_\alpha[\mathcal{C}_\epsilon(z)]$, then the quantity \eqref{EIG} is zero, and hence the degeneracy. This obviously destroys the topology of the manifold we have interest. A concrete example of this happening is when $y$ and $z$ are conjugate points on the sphere $N = S^d$. In practice, we thus should only evaluate $\mathsf{EIG}$ when two points are ``close'' in $N$.}

The assumption ``the smallest nonzero eigenvalues of $\bar{\mathcal C}_{\epsilon}(y)$ and $\bar{\mathcal C}_{\epsilon}(z)$ are of order $\epsilon^4$'' needs some discussion. The rule of thumb is that eigenvalues of $\bar{\mathcal C}_\epsilon(y)$ associated with the tangent directions are of order $1$, but the eigenvalues associated with the normal directions are of order $\epsilon^2$ or higher. The assumption that the smallest eigenvalues are of order $\epsilon^4$ means that when the principle curvature is zero, the higher order curvature is not too small. As is shown in the proof of this theorem, under this assumption, we can control the error induced by the interaction between $\iota(y)- \iota(z)$ and $\mathcal{T}_\alpha[\bar{\mathcal C}_{\epsilon}(y)]$. However, when there is an eigenvalue of even higher order, the interaction between $\iota(y)- \iota(z)$ and $\mathcal{T}_\alpha[\bar{\mathcal C}_{\epsilon}(y)]$ becomes more complicated, and a more delicate analysis involving a systematic recursive formula of Taylor expansion of  $\iota(y)- \iota(z)$ and $\mathcal{T}_\alpha[\bar{\mathcal C}_{\epsilon}(y)]$ is needed. We will explore this issue in our future work.

To sum up, in practice if we are not confident about the estimated dimension, we may want to choose a larger $\alpha$ if the dimension estimation is not accurate. Although we do not discuss it in this paper, we mention that this conservative way might not be that preferred if noise exists, since the noise will contaminate the eigenvector associated with the small eigenvalues. A more statistical approach to solve these issues will be studied in our future work.

\subsection{A special EIG setup --  covariance-corrected geodesic estimator}
When $\mathsf{\Phi}$ is identity, EIG is reduced to the ordinary manifold learning setup; that is, we have samples from the manifold and we want to estimate its geometric structure. By Lemma \ref{Lemma:1}, when two points on the smooth manifold are close enough, we can accurately estimate their geodesic distance $t$ by the ambient Euclidean distance $h=\|\iota(y)-\iota(x)\|_{\mathbb{R}^p}$ up to the {\em third} order $O(t^3)$, and the third order term is essentially the second fundamental form. We call $h$ the {\em Euclidean-distance-based geodesic distance estimator}.
In general, it is not an easy task to directly estimate the second fundamental form from the point cloud if $M'$ is not Euclidean space. However, when $M'$ is Euclidean, like the setup in Lemma \ref{Lemma:1}, the second fundamental form information could be well approximated by the local covariance matrix. We define the following projection operator associated with the local covariance matrix.

\begin{definition}
Suppose $M$ is a closed, smooth $d$ dimensional Riemannian manifold isometrically embedded in $ \mathbb{R}^p$ through $\iota$. Fix $x\in M$. For $\bar{h}>0$, let $ C_{\bar{h}}(x):= C_{x,B_{\bar{h}}^{\mathbb{R}^p}(\iota(x))\cap  \iota(M)}$ as defined in equation  (\ref{general def of covariancematrix}). Let $\lambda_{1} \geq \ldots  \geq \lambda_{p}$ be the eigenvalues of $C_{\bar{h}}(x)$, and $u_{1}, \ldots,  u_{p}$ be the corresponding normalized eigenvectors.  
Define $P^{\perp}_{\bar{h}}$ to be the orthogonal projection from $\mathbb{R}^p$ to the subspace spanned by $u_{d+1}, \ldots,  u_{p}$.
\end{definition}

The name orthogonal projection follows from Lemma \ref{prop1}, where we prove that $u_{d+1}, \ldots,  u_{p}$ deviate from an orthonormal basis of $(\iota_*T_xM)^\bot$ within a rotational error of order $O(\bar{h}^2)$.
Inspired by Lemmas \ref{prop1} and \ref{Lemma:1}, we have the following theorem. The proof is postponed to Section \ref{Section:Appendix:Thm1:proof}.

\begin{theorem} \label{Theorem1:geodesic}
Fix $x\in M$. Suppose that $y \in M$ and $\iota(y) \in B_{\bar{h}}^{\mathbb{R}^p}(\iota(x))$. We use the polar coordinate $(t,\theta)\in [0,\infty)\times S^{d-1}$ to parametrize $T_{x}M$ so that $y=\exp_{x}(\theta t)$ for $t>0$. Denote $h:=\|\iota(y)-\iota(x)\|_{\mathbb{R}^p}\leq\bar{h}$.
When $\bar{h}$ is small enough, 
\begin{equation}
\left\vert t-\left(h+\frac{\|P^{\perp}_{\bar{h}}(\iota(y)-\iota(x))\|^2_{\mathbb{R}^p}}{6h}\right) \right\vert =O(h^2\bar{h}^2)\,.
\end{equation}
\end{theorem}

Based on this Theorem, we have the following definition.
\begin{definition}
Following the notation in Theorem \ref{Theorem1:geodesic}, the {\em covariance-corrected geodesic distance estimator} between points $x,y\in M$ is defined as
\begin{equation}
h+\frac{\|P^{\perp}_{\bar{h}}(\iota(y)-\iota(x))\|^2_{\mathbb{R}^p}}{6h}\,.
\end{equation} 
\end{definition}
Clearly, the covariance-corrected geodesic distance estimator accurately estimates the geodesic distance $t$ up to the {\em fourth} order. 
It is worth mentioning that a better estimation for the geodesic distance cannot be directly achieved only with the local covariance matrix. 
Indeed, to estimate $d(x,y)$ within an error of $O(h^5)$, by Lemma \ref{Lemma:1}, we need to estimate $\nabla_{\theta}\Second_x(\theta,\theta)^\bot$ within an error of order $O(h)$. Thus, more information is needed if a more accurate geodesic distance estimate is needed.

\section{Locally Linear Embedding and a Variation by Truncation}\label{Sect:Other2}

LLE \cite{Roweis_Saul:2000} is a widely applied dimension reduction technique in the manifold learning society. Some of its theoretical properties have been explored in \cite{wu2017think}. For example, under the manifold setup, the barycentric coordinate is intimately related to the local covariance structure of the manifold. By reformulating the barycentric coordinate in the local covariance structure {framework}, the natural kernel associated with LLE is discovered, and it is very different from the kernel in graph Laplacian-based algorithms like Laplacian Eigenmaps \cite{belkin2003}, Diffusion Maps \cite{Coifman2006} or commute time embedding \cite{Qiu2007}. The regularization step in LLE is crucial in the whole algorithm, and it might provide different embeddings with different regularizations.
In this section, we connect LLE and EIG {via the local covariance structure}, and take the geometric structure to explore a variation of LLE via truncation.

\subsection{A summary of LLE algorithm}
We start with recalling the LLE algorithm in the finite sampling setup. 
Let $\mathcal{X}=\{x_i\}_{i=1}^n\subset \mathbb{R}^p$ denote a set of point cloud. 
Fix one point $x_k \in \mathcal{X}$. For $h>0$, assume $\mathcal{N}_{x_k}:=\mathcal{X} \cap \left(B^{\mathbb{R}^p}_{h}(x_k)  \setminus\{x_k\}\right)=\{x_{k,1}, \ldots, x_{k,N_k}\}$, where $h>0$. Here we take the $h$-radius ball to determine neighbors to be consistent with the following discussion. In practice we can use the $k$-nearest neighbors, where $k\in \mathbb{N}$ is determined by the user. The relationship between the $h$-radius ball and $k$-nearest neighbors is discussed in \cite[Section 5]{wu2017think}. Define the \textit{local data matrix associated with $x_k$} as 
\begin{equation*} 
G_{n,h}(x_k):=\begin{bmatrix}
| & & |\\
x_{k,1}-x_k & \cdots & x_{k,N_k}-x_k\\
| & & | 
\end{bmatrix}\in \mathbb{R}^{p \times {N_k}}.
\end{equation*}
With the local data matrix, the LLE algorithm is composed of three steps. First, for each $x_k\in \mathcal{X}$, find its barycentric coordinate associated with $\mathcal{N}_{x_k}$ by solving
\begin{equation}
w_{x_k}= \mathop{\textup{argmin}}_{w\in \mathbb{R}^{N_k},w^\top  \boldsymbol{1}_{N_k}=1} \|x_k-\sum_{j=1}^{N_k}w(j)x_{k,j} \|^2 \in \mathbb{R}^{N_k}. \label{Definition:FindBarycentric}
\end{equation}
Solving (\ref{Definition:FindBarycentric}) is equivalent to minimizing $w^\top G_{n,h}(x_k)^\top G_{n,h}(x_k)w$ over $w\in\mathbb{R}^{N_k}$ under the constraint $w^\top \boldsymbol{1}_{N_k}=1$. Since in general  $G_{n,h}(x_k)^\top G_{n,h}(x_k)$ is singular, it is recommended in \cite{Roweis_Saul:2000} to solve $(G_{n,h}(x_k)^\top G_{n,h}(x_k)+c I_{N_k\times N_k})y=\boldsymbol{1}_{N_k}$, where $c>0$ is the regularizer chosen by the user, and hence obtain 
\begin{equation}
w_{x_k}=\frac{(G_{n,h}(x_k)^\top G_{n,h}(x_k)+c I_{N_k\times N_k})^{-1}\boldsymbol{1}_{N_k}}{((G_{n,h}(x_k)^\top G_{n,h}(x_k)+c I_{N_k\times N_k})^{-1}\boldsymbol{1}_{N_k})^\top \boldsymbol{1}_N}.\label{Expansion:LLEweightedKernel0}
\end{equation}
By viewing $w_{x_k}$ as the ``affinity'' of $x_k$ and its neighbors, the second step is defining a {\em LLE matrix} $W\in \mathbb{R}^{n\times n}$  by
\begin{equation}\label{Definition:Wmatrix:LLE}
W_{k,l} = \left\{ 
\begin{array}{ll} 
w_{x_k}(j) & \mbox{if $x_l=x_{k,j} \in \mathcal{N}_{x_k}$};\\ 
0 & \mbox{otherwise}. 
\end{array} \right.
\end{equation}
Finally, to reduce the dimension of $\mathcal{X}$ or to visualize $\mathcal{X}$, it is suggested in \cite{Roweis_Saul:2000} to embed $\mathcal{X}$ into a low dimension Euclidean space
\begin{equation}
x_k\mapsto [v_1(k) ,\cdots, v_{\ell}(k)]^\top  \in \mathbb{R}^{\ell},
\end{equation}
for each $x_k\in \mathcal{X}$, where $\ell$ is the dimension of the embedded points chosen by the {user}, and $v_1, \cdots ,v_{\ell}\in \mathbb{R}^n$ are eigenvectors of $(I-W)^\top (I-W)$ corresponding to the $\ell$ smallest eigenvalues. 
Note that in general $W$ is not symmetric, and this is why the spectrum of $(I-W)^\top (I-W)$ is suggested but not that of $I-W$.

It is shown in \cite[Section 2]{wu2017think} that by taking the relationship between the {\em sample covariance matrix} $C_{n,h}(x_k) := G_{n,h}(x_k)G_{n,h}(x_k)^\top$ and the Gramian matrix $G_{n,h}(x_k)^\top G_{n,h}(x_k)$ into account, we can rewrite (\ref{Expansion:LLEweightedKernel0}) as
\begin{align}
w^\top _{x_k}&\,=\frac{\boldsymbol{1}_{N_k}^\top-\boldsymbol{1}_{N_k}^\top G_{n,h}(x_k)^\top \mathcal{I}_c(C_{n,h}(x_k) )   G_{n,h}(x_k)}
{N_k-\boldsymbol{1}_{N_k}^\top G_{n,h}(x_k)^\top \mathcal{I}_c(C_{n,h}(x_k))  G_{n,h}(x_k)\boldsymbol{1}_{N_k}}\,,
\label{Expansion:LLEweightedKernel}
\end{align} 
and view $\boldsymbol{1}_{N_k}^\top-\boldsymbol{1}_{N_k}^\top G_{n,h}(x_k)^\top \mathcal{I}_c(C_{n,h}(x_k) )   G_{n,h}(x_k)$ as the ``kernel'' associated with the LLE algorithm. It is clear that in general this kernel has negative values, and the LLE matrix is not a transition matrix.
By defining 
\begin{align}
\mathbf{T}_{n,x_k}:= \mathcal{I}_c(C_{n,h}(x_k))  G_{n,h}(x_k)\boldsymbol{1}_{N_k}\label{Definition:Tn}\,,
\end{align}    
we obtain the result claimed in \cite[Proposition 2.1]{wu2017think}.

\subsection{LLE and Mahalanobis distance in the continuous setup}

The above is for the general dataset. When the dataset $\mathcal{X}=\{\iota(x_i)\}_{i=1}^n\subset \iota(M)\subset \mathbb{R}^p$ is sampled from a manifold $M$, with (\ref{Expansion:LLEweightedKernel}) and (\ref{Definition:Tn}), the asymptotical behavior of LLE under the manifold setup is explored in \cite{wu2017think}. 
As is shown in \cite[(3.12)]{wu2017think}, in the continuous setup, the unnormalized kernel associated with LLE at $x\in M$ is 
\begin{equation}
K_h^{\texttt{LLE}}(x,y)=[1- \mathbf{T}_{\iota(x)}^\top(\iota(y)-\iota(x))]\chi_{B_{h}^{\mathbb{R}^p}(\iota(x)) \cap \iota(M)}(\iota(y)),\label{LLEKernel}
\end{equation}
where $y\in M$ and 
\begin{equation}
\mathbf{T}_{\iota(x)}:= \mathcal{I}_{c}(C_{h}(x))\big[\mathbb{E}(X-\iota(x))\chi_{B_{h}^{\mathbb{R}^p}(x)}(X)\big]    \in \mathbb{R}^p\,,\label{Definition:Tx:ContinuousCase}
\end{equation}
and hence the normalized kernel is defined by
\begin{gather}
P_h^{\texttt{LLE}}(x,y):=\frac{K_h^{\texttt{LLE}}(x,y)}{\int_M K_h^{\texttt{LLE}}(x,y)P(y)dV(y)}.
\end{gather} 
Clearly, (\ref{Expansion:LLEweightedKernel}) and (\ref{Definition:Tn}) are discretizations of $P_h^{\texttt{LLE}}(x,y)$ and $\mathbf{T}_{\iota(x)}$, and their convergence behavior could be found in \cite[(3.7) and Theorem 3.1]{wu2017think}.

There are two geometric facts we should mention. 
First, rewriting (\ref{Expansion:LLEweightedKernel0}) as (\ref{Expansion:LLEweightedKernel}) is equivalent to representing (\ref{Expansion:LLEweightedKernel0}) in the frame bundle setup, and hence an explicit form of the kernel underlying LLE.  
Second, the unnormalized kernel at $x$ involves projecting $\iota(y) - \iota(x)$ onto the vector space spanned by $u_{1},\ldots,u_{r}$, and scaling the resulting vector component-wise by  $(\lambda_{1} + c)^{-1/2}, \ldots, (\lambda_{ r} + c )^{-1/2}$.
When $r>d$, the involvement of the regularizer $c$ and the eigenvalues $\lambda_{d+1} ,\ldots, \lambda_{r}$ are associated with the normal bundle. An important intuition is that when $c\to 0$, if we put aside the expectation in $\mathbf{T}_{\iota(x)}^\top(\iota(y)-\iota(x))$ and replace $\mathbb{E}(X-\iota(x))$ by $(\iota(y)-\iota(x))$, we obtain the term $(\iota(y)-\iota(x))^\top\mathcal{T}_{r}(C_{h}(x)) (\iota(y)-\iota(x))$, which can be understood as a variation of the Mahalanobis distance. Therefore, we can interpret LLE as {{\em mixing}} the neighbor information (the 0-1 kernel in (\ref{LLEKernel}), that is, $\chi_{B_{h}^{\mathbb{R}^p}(\iota(x)) \cap \iota(M)}(\iota(y))$), and the Mahalanobis distance of neighboring points {(that is, $\mathbf{T}_{\iota(x)}^\top(\iota(y)-\iota(x))\chi_{B_{h}^{\mathbb{R}^p}(\iota(x)) \cap \iota(M)}(\iota(y))$ in (\ref{LLEKernel}))} into account to design the kernel.

Lemma~\ref{prop1} guarantees that the first $d$ eigenvalues of $C_{h}(x)$ are sufficiently large. However, the remaining $(r - d)$ non-trivial eigenvalues are small. Moreover, the remaining $(r-d)$ non-trivial eigenvalues are related to the curvature of the manifold. The importance of choosing the regularization constant $c$ has been extensively discussed in \cite[Theorem 3.2]{wu2017think}, and it is known that $c$ is to adjust the curvature information in these scaling factors $(\lambda_{d+1} + c )^{-1/2} ,\ldots, (\lambda_{r} + c )^{-1/2}$ for different purposes. To be more precise, the regularization $c$ plays a role of ``radio tuner''. The larger the $c$ is, the more enhanced the p.d.f. information will be; the smaller the $c$ is, the more dominated the curvature information will be \cite[Theorem 3.2]{wu2017think}.  
It is shown that if we want to obtain the Laplace-Beltrami operator of the Riemannian manifold, $c$ should be chosen properly so as to {\em suppress} the curvature information contained in the $(d+1)$-th to the $r$-th eigenvalues.

\subsection{A variation of LLE and EIG}
Under the manifold setup, the above discussion suggests that the curvature information associated with the $d+1,\ldots,r$ eigenvalues and eigenvectors is {\em not needed} if we want to obtain the Laplace-Beltrami operator. Therefore, an alternative to choosing a regularization parameter is a direct truncation by taking only the largest $d$ eigenvalues into account; that is, replace $\mathcal{I}_c[C_{h}(x)]$ in (\ref{LLEKernel}) by $\mathcal{T}_d[C_{h}(x)]$, and define the {\em normalized truncated LLE kernel} at $x\in M$ by
\begin{gather}
P_h^{\texttt{tLLE}}(x,y):=\frac{K_h^{\texttt{tLLE}}(x,y)}{\int_M K_h^{\texttt{tLLE}}(x,y)P(y)dV(y)}\,,
\end{gather} 
where
\begin{equation}
K_h^{\texttt{tLLE}}(x,y):=[1- \tilde{\mathbf{T}}_{\iota(x)}^\top(\iota(y)-\iota(x))]\chi_{B_{h}^{\mathbb{R}^p}(\iota(x)) \cap \iota(M)}(\iota(y)),\label{LLEKernel2}
\end{equation}
$y\in M$ and 
\begin{equation}
\tilde{\mathbf{T}}_{\iota(x)}:= \mathcal{T}_{d}(C_{h}(x))\big[\mathbb{E}(X-\iota(x))\chi_{B_{h}^{\mathbb{R}^p}(x)}(X)\big]    \in \mathbb{R}^p.\label{Definition:Tx:ContinuousCase2}
\end{equation}
Numerically when we have only finite data points, we run the same LLE algorithm, but replace $w^\top _{x_k}$ and $\mathbf{T}_{n,x_k}$ by the following terms:
\begin{align}
\tilde{w}^\top _{x_k}:=\frac{\boldsymbol{1}_{N_k}^\top-\tilde{\mathbf{T}}^\top_{n,x_k}  G_{n,h}(x_k)}{N_k-\tilde{\mathbf{T}}^\top_{n,x_k}  G_{n,h}(x_k)\boldsymbol{1}_{N_k}}\,,
\label{Expansion:tLLEweightedKernel}
\end{align} 
where
\begin{align}
\tilde{\mathbf{T}}_{n,x_k}:= \mathcal{T}_d(G_{n,h}(x_k) G_{n,h}(x_k)^\top)  G_{n,h}(x_k)\boldsymbol{1}_{N_k}\,;\label{Definition:tTn}
\end{align}    
that is, instead of running a regularized pseudo-inversion, we run a truncated inversion. {\em LLE with low-dimensional neighborhood representation} (LDR-LLE) proposed in \cite{Goldberg2008} is an algorithm related to this truncated idea.
While the geometric structure of local covariance matrix is not specifically discussed in LDR-LLE, it can be systematically studied in our framework. To simplify the nomination, we also call the above truncation scheme LDR-LLE.

Geometrically, $(\iota(y)-\iota(x))^\top \mathcal{T}_d[C_{h}(x)] (\iota(y)-\iota(x))$ evaluates geodesic distances between points in the $h$-neighbourhood of $x$ using a method directly related to both EIG in (\ref{EIG}) when $\mathsf\Phi$ is identity. With the normalized truncated LLE kernel, we can proceed with the standard LLE dimension reduction step. 
Therefore, the LDR-LLE as taking the neighbor information, $1\times \chi_{B_{h}^{\mathbb{R}^p}(\iota(x)) \cap \iota(M)}(\iota(y))$), and the EIG distance of neighboring points into account to design the kernel.

{Under the same condition specified in \cite[Theorem 3.2]{wu2017think}, LDR-LLE has the same asymptotical behavior} as that in \cite[Theorem 3.3]{wu2017think} with the properly chosen regularization; that is, 
\begin{theorem}\label{Theorem:TruncatedLLE}
Assume the setup in Section \ref{Sect:Prelims}. Take $f\in C^4(M)$. If $h=h(n)$ so that $\frac{\sqrt{\log(n)}}{n^{1/2}h^{d/2+1}}\to 0$ and $h\to 0$ when $n\to\infty$, with probability greater than $1-n^{-2}$, for all $x_k\in \mathcal{X}$ we have 
\begin{equation*}
\sum_{j=1}^{N_k}\tilde{w}_{n,x_k}(j)f(x_{k,j})=\int_M P_h^{\texttt{tLLE}}(x_k,y) f(y)P(y)dV(y)+O\left(\frac{\sqrt{\log(n)}}{n^{1/2}\epsilon^{d/2-1}}\right)
\end{equation*}
and
\begin{equation*}
\int_M P_h^{\texttt{tLLE}}(x,y) f(y)P(y)dV(y)=f(y)+\frac{h^2}{2(d+2)}\Delta  f(x)+O(h^3)\,.
\end{equation*}
\end{theorem}

As a result, with probability greater than $1-n^{-2}$, for all $x_k\in \mathcal{X}$ we have
\begin{equation*}
\frac{1}{h^2}\Big[\sum_{j=1}^{N_k}\tilde{w}_{n,x_k}(i)f(x_{k,j})-f(x_k)\Big]=\frac{1}{2(d+2)}\Delta  f(x)+O(h)+O\left(\frac{\sqrt{\log(n)}}{n^{1/2}\epsilon^{d/2+1}}\right)\,.
\end{equation*}
In other words, by the LDR-LLE, we recover the Laplace-Beltrami operator of the manifold even if the second fundamental form is non-trivial and the sampling is non-uniform. While the proof is a direct modification of the proof of \cite[Theorem 3.2]{wu2017think}, we provide the proof in Appendix \ref{Section:Appendix:Theorem3:proof} for the sake of self-containedness. We mention that the finite sample convergence has the same statement and follows exactly the same line as that of \cite[Theorem 3.1]{wu2017think}, so we skip it. 

Based on the theoretical development, we identify two benefits of LDR-LLE. First, the normal bundle information is automatically removed by the truncation, and we obtain the intrinsic geometric quantity $\Delta$. Second, the non-uniform sampling effect is automatically removed, and no kernel density estimation is needed. This comes from the fact that the EIG part of the kernel $\tilde{\mathbf{T}}_{\iota(x)}$ (or the Mahalanobis distance part of the kernel $\mathbf{T}_{\iota(x)}$) automatically provides the density information (see Lemma \ref{vector T}). Compared with DM, since we do not need normalization as is proposed in \cite{Coifman2006} to eliminate the non-uniform sampling effect, the convergence rate is faster, as is stated in \cite[Theorem 3.3]{wu2017think}. 

However, we emphasize that although this approach leads to aesthetic asymptotical properties under the manifold assumption, it may not work well when we do not have a manifold structure. For a general graph or point cloud, we found it more reliable to apply the original LLE, and a model or analysis explaining why is needed. Last but not least, note that this action requires that we know the dimension of the manifold, so even if we know the point cloud is sampled from a manifold, we need to estimate its dimension before applying the LDR-LLE. This requirement might render LDR-LLE less applicable in non-manifold data.

\section{Numerical Results}\label{Section Numerics}

We demonstrate some numerical results associated with the studied algorithms in this paper. Uniformly sample $n = 8000$ points from the logarithmic spiral $L \subset \Bbb R^2$, a 1-dim manifold parametrized as $x(s) = ( \frac{s}{\sqrt{2}} + 1 ) \cos \log ( \frac{s}{\sqrt{2}} + 1 )$ and $y(s) = ( \frac{s}{\sqrt{2}} + 1 ) \sin \log ( \frac{s}{\sqrt{2}} + 1 )$, and choose $\epsilon = 0.2$.  We fix $x_k\in L$ and plot the absolute error between the true geodesic distance and the Euclidean and covariance-corrected distances. The result is shown in Figure \ref{Fig:Examples}. Note that the two ``branches'' correspond to differences in curvature to either side of $x_k$. {In some manifold learning algorithms, knowing the local geodesic distance is not enough, and we need to know the global geodesic distance between any two points. ISOMAP \cite{tenenbaum2000} is a typical example. Usually, we estimate the global geodesic distance by applying Dijkstra's algorithm to find the shortest-path distance between all pairs, where the geodesic distance between neighboring points is estimated by the Euclidean distance. By combining the covariance-corrected geodesic distance estimator and Dijkstra's algorithm, we achieve a more accurate global geodesic distance, as is shown in Figure \ref{Fig:Examples}.}

\begin{figure} 
\centering
\includegraphics[width=0.49\textwidth]{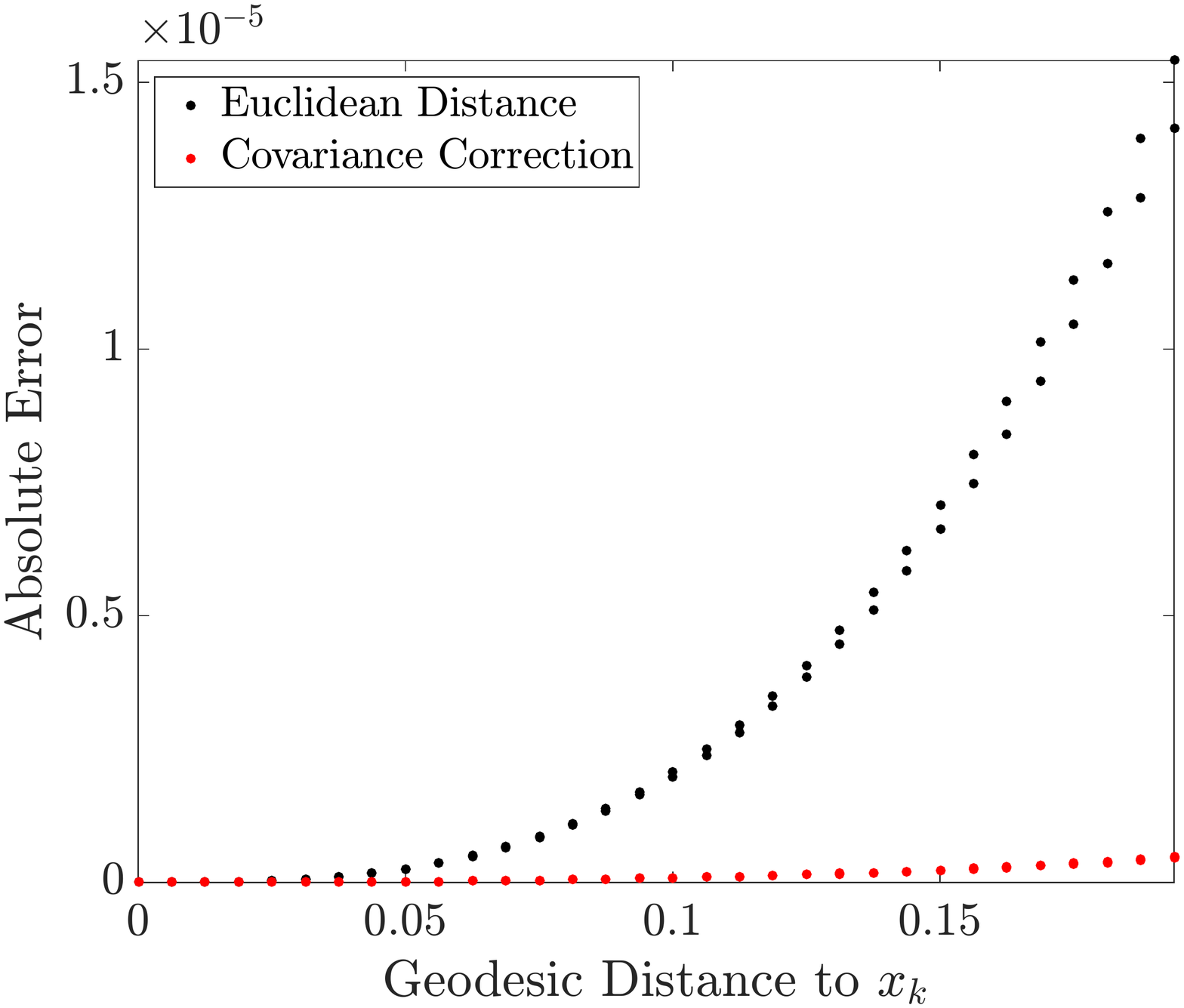}
\includegraphics[width=0.49\textwidth]{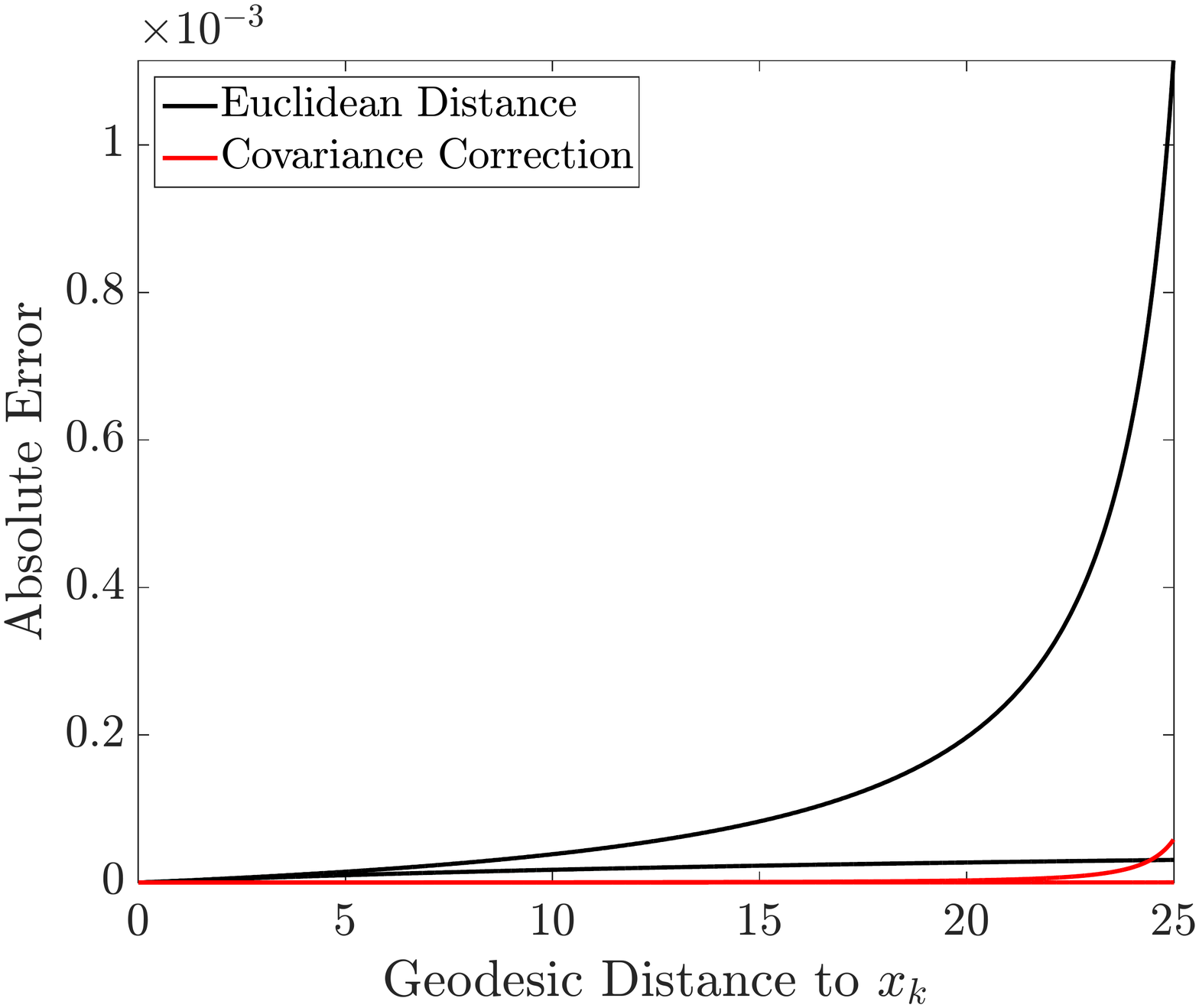}
\caption{Comparison of the Euclidean-distance-based geodesic distance estimator (black) and the covariance-corrected geodesic distance estimator (red) in the logarithmic spiral $L$. Left: local distance for $\epsilon=0.2$. Right: Global geodesic distance estimation by applying Dijkstra's algorithm with different local geodesic distance estimators.}\label{Fig:Examples}
\end{figure}

We show in Figure \ref{Fig:TruncLLE} that on $S^1\subset\mathbb{R}^2$ with a non-uniform sampling, the eigenvalues of the Laplace-Beltrami operator can be accurately recovered using the LDR-LLE. When compared with the eigenvalues obtained from the $1$-normalized DM, we see that the LDR-LLE performs better, which comes from the faster convergence rate of LDR-LLE.  We use $n = 8000$ points and we choose $h = 0.03$. The non-uniform sampling is obtained as follows. Sample $n$ points $\{\theta_i\}_{i=1}^n$ uniformly from $[0, 1]$ and set  
$x_i = ( \cos (2\pi( \theta_i + 0.3 \sin (\theta_i)) ), \sin ( 2\pi(\theta_i + 0.3\sin (\theta_i) )) ) \in S^1$, where $i = 1,\ldots, n$.

\begin{figure}
\centering
\includegraphics[width=0.5\textwidth]{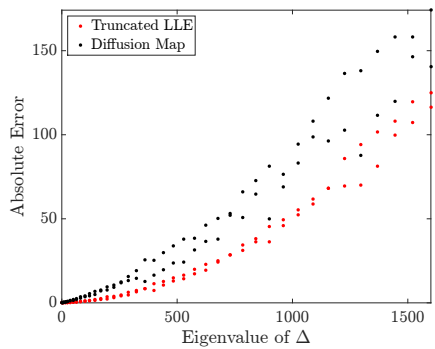}
\caption{The eigenvalues of the Laplace-Beltrami operator on $S^1$ are well-recovered by the LDR-LLE in the presence of a non-uniform sampling, compared with the $\alpha$-normalized DM.}\label{Fig:TruncLLE}
\end{figure}

\section{Conclusion}\label{Sect:Conclusion}
In this paper, we extend the knowledge of local covariance matrix, particularly the higher order expansion of the local covariance matrix, to study two commonly applied manifold learning algorithms, EIG and LLE. We provide a theoretical analysis of EIG under different situations under the manifold setup, and emphasize the importance of correctly estimating the dimension and its sensitivity to chosen parameters. The curvature effect is specifically carefully discussed.    
Under the trivial EIG setup when there is no deformation, that is, $\mathsf \Phi$ is identity, we show that the local covariance matrix structure allows us to obtain a more accurate geodesic distance for neighboring points. 
The geometric relationship between local covariance matrix and LLE leads to a natural generalization of LLE by taking EIG into account. We provide a theoretical justification of LDR-LLE and compare its pros and cons with the original LLE.

\section{Acknowledgement}
Hau-Tieng Wu thanks Professor Ronald Coifman, Professor Ronen Talmon, and Professor Stefan Steinerberger for various discussion of the topic.

\bibliographystyle{plain}
\bibliography{mybibfile}

\appendix

\renewcommand{\thelemma}{A.\arabic{lemma}}

\section{Proof of Theorem \ref{Theorem2}}\label{Section:Appendix:Theorem2:proof}

Suppose that $M$ and $N$ are $d$ dimensional closed Riemannian manifolds. $\mathsf{\Phi}:M\to N$ is a diffeomorphsim.  Moreover, suppose that $N$ is isometrically embedded in $\mathbb{R}^q$ via $\iota$.  In this section, all the derivatives of $\mathsf{\Phi}$ are calculated in normal coordinates. Hence for $x \in M$, $\nabla^k \mathsf{\Phi}(x): \underbrace{ T_x M \times \cdots \times T_xM}_{k} \rightarrow T_{\mathsf{\Phi}(x)} N$ is a multilinear map. We denote $\nabla^k \mathsf{\Phi}(x)(\theta, \cdots, \theta)= \nabla^k_{\theta \cdots \theta} \mathsf{\Phi}(x)$, for $\theta \in S^{d-1} \subset T_xM$.  For any $y \in \mathbb{R}^q$, we identify $T_{y} \mathbb{R}^q$ with $\mathbb{R}^q$. The proof of the theorem is composed of three lemmas, including Lemma \ref{Lemma2ofTheorem2}.

\begin{lemma}\label{Lemma1ofTheorem2}
Fix $x \in M$ and use polar coordinates $(t,\theta)\in [0,\infty)\times S^{d-1}$ to parametrize $T_{x}M$. For $w=\exp_x (\theta t)$ for sufficiently small $t$, let $y=\mathsf{\Phi}(x)$ and $z= \mathsf{\Phi}(w)$, then we have:
\begin{align}\label{Lemma1ofTheorem2 eqn1}
\iota (z)-&\iota (y)= \iota_*|_y \vartheta t+\frac{1}{2}\big[\Second_y(\vartheta,\vartheta)+ \iota_*|_y \nabla^2_{\theta \theta}\mathsf{\Phi}(x) \big]t^2    \\
&+\frac{1}{6}\big[\nabla_{\vartheta} \Second_y(\vartheta,\vartheta)+3\Second_y(\nabla^2_{\theta \theta} \mathsf{\Phi}(x),\vartheta)+ \iota_*|_y \nabla^3_{\theta \theta \theta}\mathsf{\Phi}(x)\big]t^3 +O(t^4), \nonumber 
\end{align}
where $\vartheta:=\nabla_\theta \mathsf{\Phi}(x)\in T_yN$, $\Second_y$ is the second fundamental form of $\iota(N)$ at $\iota(y)$,  and $\nabla \Second_y$ is the covariant derivative of the second fundamental form.
Moreover, we have
\begin{align}\label{Lemma1ofTheorem2 eqn2}
2t^2=\, & (\iota(z)-\iota(y) )^\top \iota_*|_y [\nabla\mathsf{\Phi}(x) \nabla\mathsf{\Phi}(x)^\top]^{-1} \iota_*|_y^\top (\iota(z)-\iota(y) )  \\
&+ \big(\iota(y)-\iota(z) \big)^\top \iota_*|_z [\nabla\mathsf{\Phi}(w) \nabla\mathsf{\Phi}(w)^\top]^{-1} \iota_*|_z^\top \big(\iota(y)-\iota(z) \big)+ O(t^4). \nonumber
\end{align}
\end{lemma}

Above lemma can be regarded as a generalization of Lemma \ref{Lemma:1}. Indeed, when $N=M$ and $\mathsf \Phi$ is the identity, $\nabla \mathsf \Phi$ is identity and we recover Lemma \ref{Lemma:1}.

\begin{proof}
Equation (\ref{Lemma1ofTheorem2 eqn1}) follows from a direct Taylor expansion. Since $t$ is small enough, let $\gamma:[0, t]\mapsto M$ to be the unique unit speed geodesic from $x$ to $w$; that is, $\gamma(0)=x$ and $\gamma(t)=w$. By equation (\ref{Lemma1ofTheorem2 eqn1}), we have 
\begin{align}
\iota (z)-\iota (y)=& \iota_*|_y \vartheta t+\frac{1}{2}\big[\Second_y(\vartheta,\vartheta)+ \iota_*|_y \nabla^2_{\gamma'(0) {\gamma'(0)}}\mathsf{\Phi}(x) \big]t^2+O(t^3)\,, \nonumber 
\end{align}
where $\vartheta:=\nabla_\theta \mathsf{\Phi}(x)=\nabla \mathsf{\Phi}|_x(\theta)$. Since the calculation is made in normal coordinates, we have $\iota_*|_y=\mathcal{O}_y J_{p,d}$, where $\mathcal{O}_y \in O(q)$ and $J_{q,d} \in \mathbb{R}^{q \times d}$ with $1$ on diagonal  entries and $0$ on others. Therefore, we have $\iota_*|_y^\top \iota_*|_y=I_{d \times d}$. Moreover, $\iota_*|_y^\top\Second_y(\nu, \nu)=0$ for any $\nu \in S^{d-1}$. Hence, if we multiply $[\nabla \mathsf{\Phi}(x)]^{-1}\iota_*|_y^\top$ to both sides of above equation, we have
\begin{align}
[\nabla \mathsf{\Phi}(x)]^{-1}\iota_*|_y^\top (\iota (z)-\iota (y))= \gamma'(0) t+\frac{1}{2} [\nabla \mathsf{\Phi}(x)]^{-1}\nabla^2_{\gamma'(0) \gamma'(0)}\mathsf{\Phi}(x) t^2+O(t^3). \nonumber 
\end{align}
Consequently, 
\begin{align}
&(\iota(z)-\iota(y) )^\top \iota_*|_y [\nabla\mathsf{\Phi}(x) \nabla\mathsf{\Phi}(x)^\top]^{-1} \iota_*|_y^\top (\iota(z)-\iota(y) ) \label{Lemma1ofTheorem2 eqn3} \\
=&\, (\iota(z)-\iota(y) )^\top \iota_*|_y [\nabla\mathsf{\Phi}(x)^{-1}]^\top \nabla\mathsf{\Phi}(x)^{-1} \iota_*|_y^\top (\iota(z)-\iota(y) ) \nonumber \\
=&\,  \gamma'(0)^\top \gamma'(0) t^2+\gamma'(0)^\top [\nabla \mathsf{\Phi}(x)]^{-1}\nabla^2_{\gamma'(0) \gamma'(0)}\mathsf{\Phi}(x) t^3+O(t^4) \nonumber \\
=&\, t^2+\gamma'(0)^\top [\nabla \mathsf{\Phi}(x)]^{-1}\nabla^2_{\gamma'(0) \gamma'(0)}\mathsf{\Phi}(x) t^3+O(t^4). \nonumber
\end{align}
Note that the last step follows from $ \gamma'(0)^\top \gamma'(0)=1$. 

Let $c:[0, t]\mapsto M$ be the unique geodesic  from $w$ to $x$ so that $c(0)=w$ and $c(t)=x$. Similarly, we obtain
\begin{align}
&(\iota(y)-\iota(z) )^\top \iota_*|_z [\nabla\mathsf{\Phi}(w) \nabla\mathsf{\Phi}(w)^\top]^{-1} \iota_*|_z^\top (\iota(y)-\iota(z) ) \label{Lemma1ofTheorem2 eqn4} \\
=&\, t^2+c'(0)^\top [\nabla \mathsf{\Phi}(w)]^{-1}\nabla^2_{c'(0) c'(0)}\mathsf{\Phi}(w) t^3+O(t^4) \nonumber \\
=&\, t^2-\gamma'(t)^\top [\nabla \mathsf{\Phi}(\gamma(t))]^{-1}\nabla^2_{\gamma'(t) \gamma'(t)}\mathsf{\Phi}(\gamma(t)) t^3+O(t^4) \nonumber \\
=&\, t^2-\gamma'(0)^\top [\nabla \mathsf{\Phi}(x)]^{-1}\nabla^2_{\gamma'(0) \gamma'(0)}\mathsf{\Phi}(x) t^3+O(t^4), \nonumber
\end{align}
where the second last equality comes from $w=\gamma(t)$, $c'(0)=-\gamma'(t)$ and the last equality comes from the Taylor expansion of $\gamma'(t)^\top [\nabla \mathsf{\Phi}(\gamma(t))]^{-1}\nabla^2_{\gamma'(t) \gamma'(t)}\mathsf{\Phi}(\gamma(t))$ at $t=0$. The conclusion follows from adding equations (\ref{Lemma1ofTheorem2 eqn3}) and (\ref{Lemma1ofTheorem2 eqn4}) together.
\end{proof}

\begin{proof}[Proof of Lemma \ref{Lemma2ofTheorem2}]\label{Proof:Lemma2ofTheorem2}
Fix $x\in M$ and use polar coordinates $(t,\theta)\in [0,\infty)\times S^{d-1}$ to parametrize $T_{x}M$.  Set $y'= \mathsf{\Phi} (x')$ and $x'=\exp_x(\theta t)$. 
{Let $X$ be the random variable defined on the probability space $(\Omega,\mathcal{F},\mathbb{P})$ with the range $M$, where $\mathbb{P}$ is the probability measure defined on the sigma algebra $\mathcal{F}$ in the event space $\Omega$. Define $\mathcal{P}=X_*\mathbb{P}$ to be the induced measure defined on the Borel sigma algebra on $M$. By the Radon-Nikodym theorem, $\mathcal{P}(x')=P(x') d V_M(x')$, where $P$ is the probablity density function of $X$ associated with the Riemannian volume measure, $d V_M$, defined on $M$. Furthermore, define $\mathcal{Q}$ to be the induced measure by $ \mathsf{\Phi}$ defined on the Borel sigma algebra on $N$; that is, $\mathcal{Q}=\mathsf{\Phi}_*\mathcal{P}=(\mathsf{\Phi}\circ X)_*\mathbb{P}$. By Radon-Nikodym chain rule,
\begin{align}
Q(y')=\frac{d \mathcal{Q}(y')}{d V_N(y')}=\frac{d \mathcal{P}(\mathsf{\Phi}^{-1}(y'))}{d V_M(\mathsf{\Phi}^{-1}(y'))} \frac{d V_M(\mathsf{\Phi}^{-1}(y'))}{d V_N(y')}=\frac{P(\mathsf{\Phi}^{-1}(y'))}{|\nabla\mathsf{\Phi} (\mathsf{\Phi}^{-1}(y'))|}
\end{align}
is the probability density function of the random variable $Y:=\mathsf{\Phi}\circ X$ associated with the Riemannian volume measure defined on $N$. The regularity of $Q$ is thus the same as that of $P$.}

By definition we have
\begin{equation}
\mathcal C_\epsilon(y)=\int_{E(y)} (\iota(y')-\iota(y))(\iota(y')-\iota(y))^\top Q(y')dV_N(y') \in \mathbb{R}^{q\times q}\,.
\end{equation}
By a direct Taylor expansion we have
\begin{align}
\mathcal C_\epsilon(y) =\,& \int_{B_\epsilon(x)} (\iota \circ \mathsf{\Phi} (x')-\iota \circ \mathsf{\Phi} (x))(\iota \circ \mathsf{\Phi} (x')-\iota\circ \mathsf{\Phi} (x))^\top \frac{P(x')}{ |\nabla\mathsf{\Phi} (x')|} |\nabla\mathsf{\Phi} (x')|dV_M(x') \nonumber \\
=\,& \int_{S^{d-1}}\int_{0}^{\epsilon}(\iota_*|_y \nabla \mathsf{\Phi} (x) \theta t+O(t^2))(\iota_*|_y \nabla \mathsf{\Phi} (x) \theta t+O(t^2))^\top (P(x)+O(t)) \nonumber \\
 & \quad \times (t^{d-1}+O(t^{d+1}))  dtd\theta \nonumber \\
=\,&  P(x) \int_{S^{d-1}}\int_{0}^{\epsilon}(\iota_*|_y \nabla \mathsf{\Phi} (x) \theta) (\iota_*|_y \nabla \mathsf{\Phi} (x) \theta)^\top t^{d+1}+O(t^{d+2}) dtd\theta \nonumber  \\
=\,& \frac{P(x) \epsilon^{d+2}}{d+2}[\iota_*|_y \nabla \mathsf{\Phi} (x)] \Big\{ \int_{S^{d-1}}\theta \theta^\top d\theta \Big\} [\iota_*|_y \nabla \mathsf{\Phi} (x)]^\top +O(\epsilon^{d+4}) \nonumber.
\end{align}
Note that all terms of order $t^{d+2}$  above contain a factor $\theta \theta^\top \theta$. Due to the symmetry of the sphere, $\int_{S^{d-1}} \theta \theta^\top \theta d\theta =0$. Hence, the error in the last step above is of order $O(\epsilon^{d+4})$ rather than $O(\epsilon^{d+3})$. Moreover, we have $\int_{S^{d-1}}\theta \theta^\top d\theta=\frac{|S^{d-1}|}{d} I_{d\times d}$, therefore
\begin{align}
\mathcal C_\epsilon(y) = \frac{|S^{d-1}|P(x) \epsilon^{d+2}}{d(d+2)}[\iota_*|_y\nabla \mathsf{\Phi}(x)][ {\nabla\mathsf{\Phi}(x)}^\top {\iota_*|_y}^\top] +O(\epsilon^{d+4}).
\end{align}
By a similar calculation, we have
\begin{align}
& \mathbb{E}[\chi_{E_{\epsilon}(y)}(Y)] =\int_{E(y)} Q(y')dV_N(y') \\
=\,& \int_{B_\epsilon(x)} P(x') dV_M(x') \nonumber \\
=\,& \int_{S^{d-1}}\int_{0}^{\epsilon} (P(x)+O(t))(t^{d-1}+O(t^{d+1}))  dtd\theta \nonumber \\
=\,& \frac{|S^{d-1}| P(x) \epsilon^{d}}{d}+O( \epsilon^{d+2}) \nonumber
\end{align}
Note that all the order $t^{d}$ terms above contain a factor $\theta $. Due to the symmetry of the sphere, $\int_{S^{d-1}} \theta d\theta =0$. Hence, the error in the last step above is of order $O(\epsilon^{d+2})$ rather than $O(\epsilon^{d+1})$.
The first result follows from taking the quotient of the above two equations.

By Lemma \ref{Lemma1ofTheorem2}, we have
\begin{align}
\iota \circ \mathsf{\Phi} (x')-\iota \circ \mathsf{\Phi} (x)=\,& \iota_*|_y \nabla_{\theta} \mathsf{\Phi}(x) t+\frac{1}{2}\big[\Second_y(\nabla_{\theta} \mathsf{\Phi}(x), \nabla_{\theta} \mathsf{\Phi}(x))+ \iota_*|_y \nabla^2_{\theta \theta}\mathsf{\Phi}(x) \big]t^2 +O(t^3). \nonumber 
\end{align}
If $v_1, v_2 \in (\iota_*T_{y}N)^\bot$,  since $\iota_*|_y \nabla_{\theta} \mathsf{\Phi}(x)$ and $\iota_*|_y \nabla^2_{\theta \theta}\mathsf{\Phi}(x)$ are in $\iota_*T_{y}N$, then 
\begin{align}
& v_1^\top  (\iota \circ \mathsf{\Phi} (x')-\iota \circ \mathsf{\Phi} (x))=\frac{1}{2}v_1^\top \Second_y(\nabla_{\theta} \mathsf{\Phi}(x), \nabla_{\theta} \mathsf{\Phi}(x))t^2 +O(t^3), \\
& v_2^\top  (\iota \circ \mathsf{\Phi} (x')-\iota \circ \mathsf{\Phi} (x))=\frac{1}{2}v_2^\top \Second_y(\nabla_{\theta} \mathsf{\Phi}(x), \nabla_{\theta} \mathsf{\Phi}(x))t^2 +O(t^3),
\end{align}
and we have 
\begin{align}
&v_1^\top  \mathcal C_\epsilon(y) v_2\nonumber\\
=\,& \int_{B_\epsilon(x)} v_1^\top (\iota \circ \mathsf{\Phi} (x')-\iota \circ \mathsf{\Phi} (x))(\iota \circ \mathsf{\Phi} (x')-\iota\circ \mathsf{\Phi} (x))^\top v_2 P(x') dV_M(x') \nonumber \\
=\,& \int_{S^{d-1}}\int_{0}^{\epsilon}\frac{1}{4}(v_1^\top \Second_y(\nabla_{\theta} \mathsf{\Phi}(x), \nabla_{\theta} \mathsf{\Phi}(x))t^2 +O(t^3))(v_2^\top \Second_y(\nabla_{\theta} \mathsf{\Phi}(x), \nabla_{\theta} \mathsf{\Phi}(x))t^2 +O(t^3))^\top  \nonumber \\ 
 & \quad  \times (P(x)+O(t)) (t^{d-1}+O(t^{d+1}))  dtd\theta \nonumber \\
=\,& \frac{ P(x)}{4} \int_{S^{d-1}}\int_{0}^{\epsilon}v_1^\top \Second_y(\nabla_{\theta} \mathsf{\Phi}(x), \nabla_{\theta} \mathsf{\Phi}(x)) (\Second_y(\nabla_{\theta} \mathsf{\Phi}(x), \nabla_{\theta} \mathsf{\Phi}(x)))^\top v_2 t^{d+3}+O(t^{d+4}) dtd\theta \nonumber  \\
=\,& \frac{P(x) \epsilon^{d+4}}{4(d+4)}\int_{S^{d-1}}v_1^\top \Second_y(\nabla_{\theta} \mathsf{\Phi}(x), \nabla_{\theta} \mathsf{\Phi}(x)) (\Second_y(\nabla_{\theta} \mathsf{\Phi}(x), \nabla_{\theta} \mathsf{\Phi}(x)))^\top v_2 d\theta+O(\epsilon^{d+6})\nonumber.
\end{align}
Due to the symmetry of the sphere, the error in the last step above is of order $O(\epsilon^{d+6})$ rather than $O(\epsilon^{d+5})$. The result follows from taking the quotient.
\end{proof}

To further control the influence of the truncated pseudo-inverse, we need the following lemma to further explore the higher order eigenvalue structure of the local covariance matrix. 

\begin{lemma}\label{Lemma3ofTheorem2}
Up to a rotation of $\mathbb{R}^q$, we have
\begin{align}
\bar{\mathcal C}_\epsilon(y) = \begin{bmatrix}  \frac{1}{d+2} \nabla \mathsf{\Phi}(x) {\nabla\mathsf{\Phi}(x)}^\top + O(\epsilon^2) & O(\epsilon^2) \\ O(\epsilon^2) & \frac{d}{4(d+4)} M \epsilon^2+O(\epsilon^4) \end{bmatrix},
\end{align}
where $M$ is a $(q-d)\times(q-d)$ diagonal matrix satisfying 
\begin{align}
M_{i,i}=\frac{1}{|S^{d-1}|}\int_{S^{d-1}} e_{q-d+i}^\top \Second_y(\nabla_{\theta} \mathsf{\Phi}(x), \nabla_{\theta} \mathsf{\Phi}(x)) (\Second_y(\nabla_{\theta} \mathsf{\Phi}(x), \nabla_{\theta} \mathsf{\Phi}(x)))^\top e_{q-d+i} d\theta.\label{Proof:Lemma3:Definition:M}
\end{align}
Moreover, if $M_{i,i}=0$, $e_{q-d+i}^\top \Second_y(\nabla_{\theta} \mathsf{\Phi}(x), \nabla_{\theta} \mathsf{\Phi}(x))=0$.
\end{lemma}

\begin{proof}
First, we can rotate $\iota(N)$ so that $\iota_*T_{y}N$ is generated by $\{e_1,e_2, \cdots, e_d\}$. In other words, $\iota_*|_y=\begin{bmatrix} \mathcal{O}_1 \\ 0 \end{bmatrix} \in \mathbb{R}^{q \times d}$, where $\mathcal{O}_1 \in O_{d}$. Then, by the previous lemma we have
 \begin{align}
\bar{\mathcal C}_\epsilon(y) = \begin{bmatrix}  \frac{1}{d+2} \mathcal{O}_1  \nabla \mathsf{\Phi}(x) {\nabla\mathsf{\Phi}(x)}^\top \mathcal{O}_1 ^\top + O(\epsilon^2) & O(\epsilon^2) \\ O(\epsilon^2) & \bar{M} \epsilon^2+O(\epsilon^4) \end{bmatrix},
\end{align}
where $\bar{M}$ is a $q-d$ by $q-d$ symmetric matrix. If $\mathcal{O}_2$ diagonalizes $\bar{M}$, let $\mathcal{O}=\begin{bmatrix} \mathcal{O}_1 & 0 \\ 0 & \mathcal{O}_2 \end{bmatrix}$,
and we have
 \begin{align}
\mathcal{O} ^\top \bar{\mathcal C}_\epsilon(y) \mathcal{O}= \begin{bmatrix}  \frac{1}{d+2} \nabla \mathsf{\Phi}(x) {\nabla\mathsf{\Phi}(x)}^\top + O(\epsilon^2) & O(\epsilon^2) \\ O(\epsilon^2) & \frac{d}{4(d+4)} M \epsilon^2+O(\epsilon^4) \end{bmatrix},\nonumber
\end{align}
where $M$ is a $(q-d)\times(q-d)$ diagonal matrix. And by the previous lemma
\begin{align}
 M_{i,i}=\frac{1}{|S^{d-1}|}\int_{S^{d-1}} e_{q-d+i}^\top \Second_y(\nabla_{\theta} \mathsf{\Phi}(x), \nabla_{\theta} \mathsf{\Phi}(x)) (\Second_y(\nabla_{\theta} \mathsf{\Phi}(x), \nabla_{\theta} \mathsf{\Phi}(x)))^\top e_{q-d+i} d\theta.\nonumber
\end{align}
Note that if $\theta=\sum_{j=1}^d \theta_j \partial_j\in S^{d-1} \subset T_{y}N$, then $e_{q-d+i}^\top \Second_y(\nabla_{\theta} \mathsf{\Phi}(x), \nabla_{\theta} \mathsf{\Phi}(x))=p_i(\theta_1, \theta_2, \cdots ,\theta_d)$ where $p_i$ is a quadratic form. If $ M_{i,i}=0$, then $\int_{S^{d-1}} p_i^2 d\theta=0$. Hence. $p_i=0$. The conclusion follows.
\end{proof}

We are now ready to finish the proof of Theorem \ref{Theorem2}.

\begin{proof}[Proof of Theorem \ref{Theorem2}]
Note that  $(\iota(z)-\iota(y) )^\top \mathcal{T}_\alpha [\bar{\mathcal C}_\epsilon(y)] (\iota(z)-\iota(y))$ is invariant under any rotation of $\mathbb{R}^q$. Therefore, by Lemma \ref{Lemma3ofTheorem2} we can assume that 
\begin{align}
\bar{\mathcal C}_\epsilon(y) = \begin{bmatrix}  \frac{1}{d+2} \nabla \mathsf{\Phi}(x) {\nabla\mathsf{\Phi}(x)}^\top + O(\epsilon^2) & O(\epsilon^2) \\ O(\epsilon^2) & \frac{d}{4(d+4)} M \epsilon^2+O(\epsilon^4) \end{bmatrix},
\end{align}
where $M$ is a $(q-d)\times(q-d)$ diagonal matrix shown in (\ref{Proof:Lemma3:Definition:M}).  
Suppose that $M$ has $l\geq 0$ zero diagonal entries. By assumption, we have {$q-d-l$} eigenvalues of order $\epsilon^2$ that are related to principle curvatures, and $l$ eigenvalues of order $\epsilon^4$. Write the eigendecomposition of $\bar{\mathcal C}_\epsilon(y)$ as
\begin{equation} 
\bar{\mathcal C}_\epsilon(y)=U_{\epsilon}(x)\Lambda_{\epsilon}(x)U_{\epsilon}(x)^\top\,,
\end{equation}
where $U_{\epsilon}(x)\in {O(q)}$ and $\Lambda_{\epsilon}(x)$ is a {$q\times q$} diagonal matrix.
By the perturbation argument  (see Appendix A of \cite{wu2017think} for details), $\Lambda_{\epsilon}(x)$ and $U_{\epsilon}(x)$ satisfy
\begin{equation}\label{proofoftheorem2 eigenvalue}
\Lambda_{\epsilon}(x)=\begin{bmatrix} \frac{1}{d+2} \Lambda_{1}(x) +O(\epsilon^2) & 0 & 0\\ 0 & \frac{d}{4(d+4)} \Lambda_{2}(x) \epsilon^2+O(\epsilon^4) & 0 \\0 & 0 & O(\epsilon^4)\end{bmatrix} \,,
\end{equation}
where $\Lambda_{1}(x)$ is the eigenvalue matrix of $\nabla \mathsf{\Phi}(x) {\nabla\mathsf{\Phi}(x)}^\top$, which is of order $1$, and $ \Lambda_{2}(x)$ is a $(q-d-l) \times (q-d-l)$ diagonal matrix which consists of nonzero diagonal entries of $M$, which is also of order $1$, and 
\begin{equation}\label{proofoftheorem2 eigenvector}
U_{\epsilon}(x)=\begin{bmatrix}U_{1}(x) & 0 & 0 \\ 0 & U_{2}(x) & 0 \\ 0 & 0 & U_{3}(x) \end{bmatrix}(I_{p\times p}+\epsilon^2 \mathsf{S}(x))+O(\epsilon^4), 
\end{equation}
where {$U_{1}(x)\in \mathbb{O}(d)$ }is the orthornormal eigenvector matrix of  $\mathsf{\Phi}(x) {\nabla\mathsf{\Phi}(x)}^\top$, $U_{2}(x)\in O(q-d-l)$, $U_{3}(x)\in O(l)$ and $\mathsf{S}(x)$ is anti-symmetric.\footnote{We mention that if the eigenvalues of $\mathsf{\Phi}(x) {\nabla\mathsf{\Phi}(x)}^\top$ repeat,  $U_{1}(x)$ may not be uniquely determined, or if the nonzero diagonal terms of $ \Lambda_{2}(x)$ repeat, $U_{2}(x)$ may not be identity matrix. In this case, if one wants to fully determine $U_{1}(x)$, $U_{2}(x)$ and $U_{3}(x)$, he/she needs to further explore the higher order expansion of $\bar{\mathcal C}_\epsilon(y)$. Since it is irrelevant to this proof, we refer readers with interest to Appendix A of \cite{wu2017think}.} We now finish the proof. 

\underline{\textbf{Case 1: $\alpha=d$.}} In this case, 
\begin{align}
\mathcal{T}_d[\bar{C}_\epsilon(y)]=\,&(d+2)\begin{bmatrix} U_{1}(x)+O(\epsilon^2) \\ O(\epsilon^2)\end{bmatrix} [  \Lambda_{1,\epsilon}^{-1}(x)+O(\epsilon^2)] \begin{bmatrix} U_{1}(x)^\top +O(\epsilon^2)& O(\epsilon^2)\end{bmatrix} \nonumber \\
=\, & (d+2)\begin{bmatrix} [\nabla \mathsf{\Phi}(x) {\nabla\mathsf{\Phi}(x)}^\top]^{-1} &0 \\ 0 & 0 \end{bmatrix}+ O(\epsilon^2) \nonumber \\
=\, &(d+2) \iota_*|_y [\nabla\mathsf{\Phi}(x) \nabla\mathsf{\Phi}(x)^\top]^{-1} \iota_*|_y^\top+O(\epsilon^2). \nonumber 
\end{align}
Therefore,
\begin{align}
& (\iota(z)-\iota(y) )^\top \mathcal{T}_d[\bar{\mathcal C}_\epsilon(y)] (\iota(z)-\iota(y) ) \\
=\, & (d+2) (\iota(z)-\iota(y) )^\top  \iota_*|_y [\nabla\mathsf{\Phi}(x) \nabla\mathsf{\Phi}(x)^\top]^{-1} \iota_*|_y^\top (\iota(z)-\iota(y) ) +O(\epsilon^2\|\iota(y)-\iota(z)\|_{\mathbb{R}^q}^2) \nonumber \\
=\, & (d+2)(\iota(z)-\iota(y) )^\top  \iota_*|_y [\nabla\mathsf{\Phi}(x) \nabla\mathsf{\Phi}(x)^\top]^{-1} \iota_*|_y^\top (\iota(z)-\iota(y) ) +O(\epsilon^2t^2)\,, \nonumber
\end{align}
where the last step follows from Lemma \ref{Lemma1ofTheorem2} that $O(\|\iota(y)-\iota(z)\|_{\mathbb{R}^q}^2)=O(t^2)$.
Similarly, 
\begin{align}
& (\iota(y)-\iota(z) )^\top \mathcal{T}_d[\bar{\mathcal C}_\epsilon(z)] (\iota(y)-\iota(z) ) \\
=\, &(d+2)  (\iota(y)-\iota(z) )^\top  \iota_*|_z [\nabla\mathsf{\Phi}(w) \nabla\mathsf{\Phi}(w)^\top]^{-1} \iota_*|_z^\top (\iota(y)-\iota(z) ) +O(\epsilon^2t^2). \nonumber
\end{align}
By Lemma \ref{Lemma1ofTheorem2}, we have
\begin{align}
\mathsf{EIG}_{d}^2(y,z)=&\, (\iota(z)-\iota(y) )^\top \Big[ \frac{\mathcal{T}_d[\bar{\mathcal C}_\epsilon(y)]+\mathcal{T}_d[\bar{\mathcal C}_\epsilon(z)]}{2}\Big](\iota(z)-\iota(y) )\nonumber\\
=&\,(d+2)t^2 +O(t^2\epsilon^2)+O(t^4)\,. \nonumber 
\end{align}

\underline{\textbf{Preparation for $\alpha \neq d$.}} 
To show the case when $\alpha \not = d$, we need some preparations.
Let $\lambda_1 \geq \lambda_2 \geq \cdots \lambda_q$ denote the eigenvalues of $\bar{\mathcal C}_\epsilon(y)$, and let $\{ u_i \}$ be the corresponding normalized eigenvectors. Note that $u_i$ is the $i$-th column of $\begin{bmatrix} U_{1}(x) \\ 0\\ 0\end{bmatrix}+O(\epsilon^2)$ when $1\leq i\leq d$ in (\ref{proofoftheorem2 eigenvector}). Based on (\ref{proofoftheorem2 eigenvector}), for $1\leq i \leq d$, the first $d$ entries of $u_i$ are of order $1$ and the other entries are of order $\epsilon^2$. Note that the first $d$ entries of $u_i$ are associated with the tangent space and the others are associated with the normal space. For $d+1\leq i \leq q-l$, the tangent components of $u_i$ (the first $d$ entries) are of order $\epsilon^2$ and the normal components of $u_i$ (the remaining $p-d$ entries) are of order $1$. On the other hand, based on Lemma \ref{Lemma1ofTheorem2}, the tangent component of $\iota(z)-\iota(y)$ is of order $t$  and the normal component of $\iota(z)-\iota(y)$ is of order $t^2$. Following the notation in Lemma \ref{Lemma1ofTheorem2} and putting the above together, for $1\leq i\leq p-l$, we have 
\begin{align}
(\iota(z)-\iota(y) )^\top u_i=O(t) & \quad  \mbox{for $1 \leq i \leq d$}, \label{proofoftheorem2 1tod} \\ 
(\iota(z)-\iota(y) )^\top u_i=O(t^2+t\epsilon^2) & \quad \mbox{for $d+1 \leq i \leq q-l$}.\label{proofoftheorem2 d+1toq-l}
\end{align}
For $q-l+1 \leq i \leq q$, we have $u_i=\bar{u}_{i}+O(\epsilon^2)$, where $\bar{u}_{i}$ is $(i-q-l)$'s column of $\begin{bmatrix} 0 \\ 0 \\ U_{3}(x)\end{bmatrix}$ shown in (\ref{proofoftheorem2 eigenvector}).
By Lemma \ref{Lemma1ofTheorem2}, $\bar{u}_{i}^\top (\iota(z)-\iota(y) )$ becomes
\begin{align}
\bar{u}_{i}^\top  \iota_*|_y \nabla_{\theta} \mathsf{\Phi}(x) t+\frac{1}{2}\big[\bar{u}_{i}^\top \Second_y(\nabla_{\theta} \mathsf{\Phi}(x), \nabla_{\theta} \mathsf{\Phi}(x))+\bar{u}_{i}^\top  \iota_*|_y \nabla^2_{\theta \theta}\mathsf{\Phi}(x) \big]t^2 +O(t^3). \nonumber 
\end{align}
Note that  $\bar{u}_{i}^\top  \iota_*|_y \nabla_{\theta} \mathsf{\Phi}(x)=0$ and $\bar{u}_{i}^\top  \iota_*|_y \nabla^2_{\theta \theta}\mathsf{\Phi}(x)=0$, since $\bar{u}_{i}$ is in the normal direction. By Lemma \ref{Lemma3ofTheorem2}, we have 
$\bar{u}_{i}^\top \Second_y(\nabla_{\theta} \mathsf{\Phi}(x), \nabla_{\theta} \mathsf{\Phi}(x))=0$. 
Therefore $\bar{u}_{i}^\top (\iota(z)-\iota(y) )=O(t^3)$. 
Since $u_i-\bar{u}_{i}$ is of order $O(\epsilon^2)$,  $\iota(z)-\iota(y)$ is of order $O(t)$, we have $(u_i-\bar{u}_{i})^\top (\iota(z)-\iota(y) )=O(t \epsilon^2)$.  Putting the above together gives 
\begin{align}
(\iota(z)-\iota(y) )^\top u_i=O(t^3+t\epsilon^2)  \quad \mbox{for $q-l+1 \leq i \leq q$}.\label{proofoftheorem2 q-l+1toq}
\end{align}

With (\ref{proofoftheorem2 1tod}), (\ref{proofoftheorem2 d+1toq-l}) and (\ref{proofoftheorem2 q-l+1toq}), we can finish the proof for $\alpha\neq d$. 

\underline{\textbf{Case 2: $1 \leq \alpha <d$.}} In this case, we have
\begin{align}
&(\iota(z)-\iota(y) )^\top \big(\mathcal{T}_\alpha[\bar{\mathcal C}_\epsilon(y)]-\mathcal{T}_d[\bar{\mathcal C}_\epsilon(y)]\big)(\iota(z)-\iota(y) )\\
= \,&(\iota(z)-\iota(y) )^\top \Big[\sum_{j=\alpha+1}^d \frac{u_j u_j^\top}{\lambda_j}\Big](\iota(z)-\iota(y) ). \nonumber
\end{align}
Recall that $\lambda_j$ is of order 1 for $1 \leq j \leq d$ and also (\ref{proofoftheorem2 1tod}).
As a result, 
\begin{align}
(\iota(z)-\iota(y) )^\top \mathcal{T}_\alpha[\bar{\mathcal C}_\epsilon(y)](\iota(z)-\iota(y) )=(\iota(z)-\iota(y) )^\top \mathcal{T}_d[\bar{\mathcal C}_\epsilon(y)](\iota(z)-\iota(y) )+O(t^2).\nonumber
\end{align}
Similarly, 
\begin{align}
(\iota(z)-\iota(y) )^\top \mathcal{T}_\alpha[\bar{\mathcal C}_\epsilon(z)](\iota(z)-\iota(y) )=(\iota(z)-\iota(y) )^\top \mathcal{T}_d[\bar{\mathcal C}_\epsilon(z)](\iota(z)-\iota(y) )+O(t^2).\nonumber
\end{align}
We then achieve the claim
\begin{align}
\mathsf{EIG}_{\alpha}^2(y,z)=(d+2)t^2 +O(t^2)\,.
\end{align}

We now show a special case when $1\leq \alpha<d$. Let the geodesic distance between $y$ and $z$ be $d(y,z)$, then by Lemma \ref{Lemma1ofTheorem2}, we have $d(y,z)=\|\iota(y)-\iota(z)\|_{ \mathbb{R}^q}+O(\|\iota(y)-\iota(z)\|_{ \mathbb{R}^q}^3)=O(t)$.
In a special case when $\iota_{*}|_y \vartheta(y) \in V_\alpha(y)$, where $V_\alpha(y)$ is the subspace of $\iota_*T_yN$ generated by the first $\alpha$ eigenvectors of $\bar{\mathcal{C}}_\epsilon(y)$, we have
\begin{align}
&(\iota(z)-\iota(y) )^\top \big(\mathcal{T}_\alpha[\bar{\mathcal C}_\epsilon(y)]-\mathcal{T}_d[\bar{\mathcal C}_\epsilon(y)]\big)(\iota(z)-\iota(y) )\\
= \,&(\iota(z)-\iota(y) )^\top \Big[\sum_{j=\alpha+1}^d \frac{u_j u_j^\top}{\lambda_j}\Big](\iota(z)-\iota(y) )=O(d(y,z)^4)=O(t^4). \nonumber
\end{align}
If furthermore $\iota_{*}|_z \vartheta(z) \in V_\alpha(z)$, where $V_\alpha(z)$ is the subspace of $\iota_*T_zN$ generated by the first $\alpha$ eigenvectors of $\bar{\mathcal{C}}_\epsilon(z)$, we have
\begin{align}
&(\iota(z)-\iota(y) )^\top \big(\mathcal{T}_\alpha[\bar{\mathcal C}_\epsilon(z)]-\mathcal{T}_d[\bar{\mathcal C}_\epsilon(z)]\big)(\iota(z)-\iota(y) )=O(t^4)
\end{align}
and hence
\begin{align}
\mathsf{EIG}_{\alpha}^2(y,z)=(d+2)t^2 +O(t^2\epsilon^2)+O(t^4)\,.
\end{align}

\underline{\textbf{Case 3: $\alpha >d$.}}  In this case, we have $0\leq l_y\leq q-d$. 
By a similar calculation, we have
\begin{align}
&(\iota(z)-\iota(y) )^\top \big(\mathcal{T}_\alpha[\bar{\mathcal C}_\epsilon(y)]-\mathcal{T}_d[\bar{\mathcal C}_\epsilon(y)]\big)(\iota(z)-\iota(y) )\nonumber\\
=&\, (\iota(z)-\iota(y) )^\top \Big[\sum_{j=d+1}^{q-l_y} \frac{u_j u_j^\top}{\lambda_j}+\sum_{j=q-l_y+1}^{\alpha} \frac{u_j u_j^\top}{\lambda_j}\Big](\iota(z)-\iota(y) ) \,.\nonumber 
\end{align}
Since $\lambda_j$ is of order $\epsilon^2$ for $d+1 \leq j \leq q-l_y$, by (\ref{proofoftheorem2 d+1toq-l})
\begin{align}
 (\iota(z)-\iota(y) )^\top \sum_{j=d+1}^{q-l_y} \frac{u_j u_j^\top}{\lambda_j}(\iota(z)-\iota(y) )=O((t^2/\epsilon+t\epsilon)^2). \nonumber
\end{align}
When $q-l_y+1 \leq j \leq \alpha$, $\lambda_j$ is of order $\epsilon^4$ by assumption. By (\ref{proofoftheorem2 q-l+1toq})
\begin{align}
 (\iota(z)-\iota(y) )^\top \sum_{j=q-l_y+1}^{\alpha} \frac{u_j u_j^\top}{\lambda_j}(\iota(z)-\iota(y) )=O((t^3/\epsilon^2+t)^2). \nonumber
\end{align}
Hence, {when $\alpha\geq q-l_y+1$,}
\begin{align}
(\iota(z)-\iota(y) )^\top \big(\mathcal{T}_\alpha[\bar{\mathcal C}_\epsilon(y)]-\mathcal{T}_d[\bar{\mathcal C}_\epsilon(y)]\big)(\iota(z)-\iota(y) )=O((t^2/\epsilon+t\epsilon)^2+(t^3/\epsilon^2+t)^2); \nonumber
\end{align}
{otherwise we have
\begin{align}
(\iota(z)-\iota(y) )^\top \big(\mathcal{T}_\alpha[\bar{\mathcal C}_\epsilon(y)]-\mathcal{T}_d[\bar{\mathcal C}_\epsilon(y)]\big)(\iota(z)-\iota(y) )=O((t^2/\epsilon+t\epsilon)^2)\,.\nonumber
\end{align}}
By the same argument, {when $\alpha\geq q-l_z+1$, we have}
\begin{align}
(\iota(z)-\iota(y) )^\top \big(\mathcal{T}_\alpha[\bar{\mathcal C}_\epsilon(z)]-\mathcal{T}_d[\bar{\mathcal C}_\epsilon(z)]\big)(\iota(z)-\iota(y) )=O((t^2/\epsilon+t\epsilon)^2+(t^3/\epsilon^2+t)^2); \nonumber
\end{align}
{otherwise we have
\begin{align}
(\iota(z)-\iota(y) )^\top \big(\mathcal{T}_\alpha[\bar{\mathcal C}_\epsilon(z)]-\mathcal{T}_d[\bar{\mathcal C}_\epsilon(z)]\big)(\iota(z)-\iota(y) )=O((t^2/\epsilon+t\epsilon)^2)\,.\nonumber
\end{align}}
Summing above two equations together {leads the fact that when $l_y>0$ and $\alpha\geq q-l_y+1$ hold, or $l_z>0$ and $\alpha\geq q-l_z+1$ hold, we have}
\begin{align}
\mathsf{EIG}_{\alpha}^2(y,z)=(d+2)t^2+O((t^2/\epsilon+t\epsilon)^2+(t^3/\epsilon^2+t)^2)\,;
\end{align}
{otherwise we have
\begin{align}
\mathsf{EIG}_{\alpha}^2(y,z)=(d+2)t^2+O((t^2/\epsilon+t\epsilon)^2)\,.\nonumber
\end{align}}
Set $t=\epsilon^\beta$, where $\beta\geq 1$. 
By a straightforward calculation, {when $l_y=0$ or $\alpha\leq q-l_y$ holds, and $l_z=0$ or $\alpha\leq q-l_z$ holds}, $t^2$ dominates $O((t^2/\epsilon+t\epsilon)^2$ asymptotically when $\beta>1$; otherwise $t^2$ and $O((t^2/\epsilon+t\epsilon)^2+(t^3/\epsilon^2+t)^2)$ are asymptotically of the same order. 
We thus conclude the claim. 
\end{proof}

\begin{remark}
We would like to make a comment {when $l_y>0$ and $\alpha> q-l_y$ hold, or $l_z>0$ and $\alpha> q-l_z$ hold}. To simplify the discussion, assume $l_y=l_z=l>0$ and $\alpha> q-l$.
For $u\in\mathbb{R}^q$, we write $u=[u_1^\top, \,u_2^\top, \,u_3^\top]^\top$, where $u_1 \in \mathbb{R}^d$. $u_2 \in \mathbb{R}^{q-d-l}$ and  $u_3 \in \mathbb{R}^{l}$. 
As we stated in the proof of equation (\ref{proofoftheorem2 d+1toq-l}), $u_j=[O(\epsilon^2),\,O(1), \,O(\epsilon^2)]^\top$ for $j=d+1, \cdots, q-l$. On the other hand, $\iota(z)-\iota(y)=[O(t),O(t^2),O(t^3)]^\top$. Here, based on the structure of $\bar{\mathcal C}_\epsilon(y)$, the first $d$ components of $\iota(z)-\iota(y)$ are in the tangent direction of $\iota(N)$ at $\iota(y)$, and they are of order $O(t)$. 
The next $q-d-l$ components are in the direction of the second fundamental form of $\iota$ at $\iota(y)$, and they are of order $O(t^2)$. 
The last $l$ components are in the normal direction and perpendicular to the second fundamental form of $\iota$ at $\iota(y)$, and they are of order $O(t^3)$. 
When we calculate the product $(\iota(z)-\iota(y) )^\top u_j$, the products of components in the tangent directions and the products of components in the normal directions cannot be canceled in general. 
The argument holds for equation (\ref{proofoftheorem2 q-l+1toq}). Hence, the order estimation for Case $\alpha >d$ cannot be improved for an arbitrary vector $\iota(z)-\iota(y)$ without imposing more conditions. In other words, an expansion of $u_j$ or $\iota(z)-\iota(y)$ into higher orders might not improve the results. 
\end{remark}

\section{Proof of Theorem \ref{Theorem1:geodesic}} \label{Section:Appendix:Thm1:proof}

By Lemma \ref{prop1}, we have
\begin{equation}
U_{\bar{h}}(x)=\begin{bmatrix}U_1 & 0 \\ 0 & U_2\end{bmatrix}+O(\bar{h}^2), \nonumber
\end{equation}
where {$U_1\in \mathbb{O}(d)$ and $U_2\in \mathbb{O}(p-d)$.}
By Equation~\eqref{oldLemma1},
\begin{equation}
\iota(y)-\iota(x)=(\iota_{*}\theta) t+\frac{\Second_{x}(\theta,\theta)}{2} t^2+O(t^3). \nonumber
\end{equation}
By Equation \eqref{oldLemma2a}, 
\begin{equation}
t=h+\frac{\|\Second_{x}(\theta,\theta)\|^2}{24}  h^3 +O(h^4).\nonumber
\end{equation}
Hence
\begin{equation}
\iota(y)-\iota(x)=(\iota_{*}\theta) h+\frac{\Second_{x}(\theta,\theta)}{2} h^2+O(h^3). \nonumber
\end{equation}
Note that $(\iota_{*}\theta) h$ is in $\iota_*T_{x}M$. Therefore 
\begin{align}
P^{\perp}_{h}(\iota(y)-\iota(x))&=\frac{\Second_{x}(\theta,\theta)}{2} h^2+O(h^3+h\bar{h}^2)=\frac{\Second_{x}(\theta,\theta)}{2} h^2+O(h\bar{h}^2), \nonumber
 \end{align}
where we use the fact that $h < \bar{h}$ in the last step. Hence,
\begin{align}
\frac{\|P^{\perp}_{h}(\iota(y)-\iota(x))\|^2_{\mathbb{R}^p}}{6h}&= \frac{\|\Second_{x}(\theta,\theta)\|^2}{24}  h^3+O(h^2\bar{h}^2). \nonumber
\end{align}
The conclusion follows.\qed

\section{Proof of Theorem \ref{Theorem:TruncatedLLE}}\label{Section:Appendix:Theorem3:proof}

The proof  of Theorem \ref{Theorem:TruncatedLLE} consists of two steps, the variance analysis and bias analysis.
The variance analysis of the LDR-LLE is similar to Case 0 in \cite[Theorem 3.1]{wu2017think}, so we only provide the result and the proof is omitted.
\begin{proposition} 
Fix $f \in C(\iota(M))$. Suppose $h=h(n)$ so that $\frac{\sqrt{\log(n)}}{n^{1/2}h^{d/2+1}}\to 0$ and $h\to 0$ as $n\to \infty$. With probability greater than $1-n^{-2}$, for all $x_k\in\mathcal{X}$,
\begin{align}
\sum_{j=1}^N\tilde{w}_{n,x_k}(i)f(x_{k,j})=\int_M P_h^{\texttt{tLLE}}(x_k,y) f(y)P(y)dV(y)+O\Big(\frac{\sqrt{\log (n)}}{n^{1/2}h^{d/2-1}}\Big)\nonumber.
\end{align}
\end{proposition}

The bias analysis part of Theorem \ref{Theorem:TruncatedLLE} depends on the following two technical lemmas. We use the following notations to simplify the proof. For $p,d\in \mathbb{N}$ so that $d \leq p$, denote $J_{p,d}\in \mathbb{R}^{p\times d}$ so that the $(i,i)$ entry is $1$ for $i=1,\ldots,d$, and zeros elsewhere.  For $v\in\mathbb{R}^p$, 
\begin{align} \label{vectornotation}
v=[\![v_1,\,v_2]\!]\in \mathbb{R}^p\,,
\end{align}
where $v_1\in \mathbb{R}^{d}$ forms the first $d$ coordinates of $v$ and $v_2\in\mathbb{R}^{p-d}$ forms the last $p-d$ coordinates of $v$. 
Thus, by a proper translation and rotation of $\iota(M)$ in $\mathbb{R}^p$ so that $\iota(x)=0$ and $\iota_*T_x$ occupies the first $d$ axes of $\mathbb{R}^p$, for $v=[\![v_1,\,v_2]\!]\in T_{\iota(x)}\mathbb{R}^p$, $v_1=J_{p,d}^\top v$ is tangential to $\iota_*T_xM$ and $[\![0,\,v_2]\!]$ is normal to $\iota_*T_xM$. 

\begin{lemma}(\cite[Lemma~B.5]{wu2017think})\label{prep lemma}
Fix $x \in  M$ and $f\in C^3(M)$. When $h>0$ is sufficiently small, the following expansions hold.
\begin{enumerate} 
\item The scalar $\mathbb{E}[\chi_{B_{h}^{\mathbb{R}^p}(\iota(x))}(X)]$ satisfies
\begin{equation} 
\mathbb{E}[\chi_{B_{h}^{\mathbb{R}^p}(\iota(x))}(X)]= P(x)\frac{|S^{d-1}|}{d}h^d+O(h^{d+2}).\nonumber
\end{equation} 
\item The scalar $\mathbb{E}[(f(X)-f(x))\chi_{B_{h}^{\mathbb{R}^p}(\iota(x))}(X)] $ satisfies 
\begin{align}
& \mathbb{E}[(f(X)-f(x))\chi_{B_{h}^{\mathbb{R}^p}(\iota(x))}(X)] \nonumber\\
=&\, \frac{|S^{d-1}|}{d(d+2)}\Big[\frac{1}{2}P(x)\Delta f(x)+\nabla f(x)\cdot \nabla P(x)\Big]h^{d+2}+O(h^{d+3}).\nonumber
\end{align} 
\item The vector $\mathbb{E}[(X-\iota(x))\chi_{B_{h}^{\mathbb{R}^p}(\iota(x))}(X)]$ satisfies
\begin{align}
&\mathbb{E}[(X-\iota(x))\chi_{B_{h}^{\mathbb{R}^p}(\iota(x))}(X)]\nonumber\\
=&\,\Big[\!\!\Big[\frac{|S^{d-1}|}{d(d+2)} J_{p,d}^{\top} \iota_*\nabla {P}(x)h^{d+2}+O(h^{d+3}),O(h^{d+2})\Big]\!\!\Big]\,. \nonumber
\end{align}
\item The vector $\mathbb{E}[(X-\iota(x))(f(X)-f(x))\chi_{B_{h}^{\mathbb{R}^p}(\iota(x))}(X)]$ satisfies
\begin{align}
&\mathbb{E}[(X-\iota(x))(f(X)-f(x))\chi_{B_{h}^{\mathbb{R}^p}(\iota(x))}(X)]\nonumber\\
=&\,\Big[\!\!\Big[\frac{|S^{d-1}|}{d(d+2)} J_{p,d}^{\top} P(x)\iota_*\nabla {f}(x) h^{d+2}+O(h^{d+3}),O(h^{d+4})\Big]\!\!\Big]\,.\nonumber
\end{align}
\end{enumerate}
\end{lemma}
The proof of the above Lemma can be found in \cite[Lemma~B.5]{wu2017think}.
The next lemma describes the vector $\mathbf{T}$ when $h$ is sufficiently small. The main significance is that the vector $\tilde{\mathbf{T}}_{\iota(x)}$ almost recovers $\nabla \log P(x)$ in the tangent direction. 
\begin{lemma}\label{vector T}
Fix $x\in M$. When $h$ is sufficiently small, we have
\[
\tilde{\mathbf{T}}_{\iota(x)}=\Big[\!\!\Big[\frac{J_{p,d}^\top \iota_*\nabla {P}(x)}{P(x)}+O(h^2),\,O(h^2)\Big]\!\!\Big].
\]
\end{lemma}

\begin{proof}
Note that 
\begin{align}
\tilde{\mathbf{T}}_{\iota(x)}:=\mathcal{T}_d [C_{h}(x)] \big[\mathbb{E}(X-\iota(x))\chi_{B_{h}^{\mathbb{R}^p}(\iota(x))}\big]= \sum_{i=1}^d\frac{u_i u_i^\top \big[\mathbb{E}(X-\iota(x))\chi_{B_{h}^{\mathbb{R}^p}(\iota(x))}\big]}{\lambda_i}\,,\nonumber
\end{align}
where $u_i$ and $\lambda_i$ form the $i$-th eigenpair of $C_{h}(x)$.
By Lemma \ref{prop1}, $u_i=\begin{bmatrix}U_1J_{p,d}^\top e_i+O(h^2)\\  O(h^2) \end{bmatrix}$, $i=1,\ldots,d$ {and $U_1\in \mathbb {O} (d)$,} and $\lambda_i=\frac{|S^{d-1}|  {P}(x)}{d(d+2)}h^{d+2}+O(h^{d+4})$. By Lemma \ref{prep lemma} , we have for $1 \leq i \leq d$
\begin{align} 
&u_i^\top \mathbb{E}[(X-\iota(x))\chi_{B_{h}^{\mathbb{R}^p}(x_k)}(X)]   \nonumber\\
=&\, \frac{|S^{d-1}|}{d+2}\Big[\!\!\Big[ \frac{J_{p,d}^\top \iota_*\nabla {P}(x)}{d}h^{d+2} +O(h^{d+4}),\,O(h^{d+2})\Big]\!\!\Big] \cdot \Big[\!\!\Big[U_1J_{p,d}^\top e_i+O(h^2), O(h^2)\Big]\!\!\Big] \nonumber \\
=&\, \frac{|S^{d-1}|}{d(d+2)}  (\iota_*\nabla {P}(x))^\top  J_{p,d} U_1J_{p,d}^\top e_i h^{d+2}+O(h^{d+4}) \nonumber\,.
\end{align} 
Thus, for $1 \leq i \leq d$,
\begin{align} 
\frac{u_i^\top \mathbb{E}[(X-\iota(x))\chi_{B_{h}^{\mathbb{R}^p}(x_k)}(X)]  }{\lambda_i} &= \frac{\frac{|S^{d-1}|}{d(d+2)}  (\iota_*\nabla {P}(x))^\top  J_{p,d} U_1J_{p,d}^\top e_i h^{d+2}+O(h^{d+4}) }{\frac{|S^{d-1}|  {P}(x)}{d(d+2)}h^{d+2}+O(h^{d+4})}\nonumber\\
&=\frac{ (\iota_*\nabla {P}(x))^\top  J_{p,d} U_1J_{p,d}^\top e_i }{{P}(x)}+O(h^2)\,.\nonumber
\end{align} 
Hence, we have
\begin{align}
\tilde{\mathbf{T}}_{\iota(x)} &=\sum_{i=1}^d\frac{u_i^\top \mathbb{E}[(X-\iota(x))\chi_{B_{h}^{\mathbb{R}^p}(x_k)}(X)] }{\lambda_i}u_i\nonumber \\
&=\sum_{i=1}^d\Big(\frac{ (\iota_*\nabla {P}(x))^\top  J_{p,d} U_1J_{p,d}^\top e_i }{{P}(x)}+O(h^2)\Big)\Big[\!\!\Big[U_1J_{p,d}^\top e_i+O(h^2),\, O(h^2)\Big]\!\!\Big]\nonumber\\
&=\Big[\!\!\Big[\frac{J_{p,d}^\top \iota_*\nabla {P}(x)}{{P}(x)}+O(h^2),\,O(h^2)\Big]\!\!\Big],\nonumber
\end{align}
where the last step follows from the fact that {$U_1\in \mathbb{O}(d)$.}
\end{proof}

We are ready to prove Theorem \ref{Theorem:TruncatedLLE}.
\begin{proof}[Proof of Theorem \ref{Theorem:TruncatedLLE}]
Note that
\begin{align}
&\int_M P_{\texttt{tLLE}}(x,y) f(y)P(y)dV(y)-f(x)\nonumber\\
=&\,\frac{\mathbb{E}[(f(X)-f(x))\chi_{B_{h}^{\mathbb{R}^p}(\iota(x))}(X)] -\tilde{\mathbf{T}}_{\iota(x)}^\top \mathbb{E}[(X-\iota(x))(f(X)-f(x))\chi_{B_{h}^{\mathbb{R}^p}(\iota(x))}(X)]}{\mathbb{E}[\chi_{B_{h}^{\mathbb{R}^p}(\iota(x))}(X)]-\tilde{\mathbf{T}}_{\iota(x)}^\top \mathbb{E}[(X-\iota(x))\chi_{B_{h}^{\mathbb{R}^p}(\iota(x))}(X)]}\,.\nonumber
\end{align}
By Lemma \ref{prep lemma} and Lemma \ref{vector T}, we have 
\begin{align}
\tilde{\mathbf{T}}_{\iota(x)}^\top \mathbb{E}[(X-\iota(x))\chi_{B_{h}^{\mathbb{R}^p}(\iota(x))}(X)]=O(h^{d+2}),
\end{align}
and
\begin{align}
\tilde{\mathbf{T}}_{\iota(x)}^\top \mathbb{E}[(X-\iota(x))(f(X)-f(x))\chi_{B_{h}^{\mathbb{R}^p}(\iota(x))}(X)]=\frac{|S^{d-1}|}{d(d+2)} \nabla P(x) \cdot \nabla {f}(x) h^{d+2}+O(h^{d+3}).\nonumber
\end{align}
Hence,
\begin{align}
\mathbb{E}[\chi_{B_{h}^{\mathbb{R}^p}(\iota(x))}(X)]-\mathbf{T}_{\iota(x)}^\top \mathbb{E}[(X-\iota(x))\chi_{B_{h}^{\mathbb{R}^p}(\iota(x))}(X)]=P(x)\frac{|S^{d-1}|}{d}h^d+O(h^{d+2}),\nonumber
\end{align}
and
\begin{align}
& \mathbb{E}[(f(X)-f(x))\chi_{B_{h}^{\mathbb{R}^p}(\iota(x))}(X)] -\mathbf{T}_{\iota(x)}^\top \mathbb{E}[(X-\iota(x))(f(X)-f(x))\chi_{B_{h}^{\mathbb{R}^p}(\iota(x))}(X)]\nonumber \\
=& \frac{|S^{d-1}|}{d(d+2)}\Big[\frac{1}{2}P(x)\Delta f(x)+\nabla f(x)\cdot \nabla P(x)\Big]h^{d+2}-\frac{|S^{d-1}|}{d(d+2)} \nabla P(x) \cdot \nabla {f}(x) h^{d+2}+O(h^{d+3}) \nonumber \\
=& \frac{|S^{d-1}|}{2d(d+2)}P(x)\Delta f(x)h^{d+2}+O(h^{d+3}). \nonumber
\end{align}
If we combine the above two equations, we conclude the claim that
\begin{equation}
\int_M P_{\texttt{tLLE}}(x,y) f(y)P(y)dV(y)-f(x)=\frac{1}{2(d+2)}\Delta  f(x) h^2+O(h^3).\nonumber
\end{equation}
\end{proof}

\end{document}